%% file: dompn_arxiv.tex
\documentclass[a4paper, 10pt]{amsart}

\usepackage{amssymb}
\usepackage{amsmath,amsthm}
\usepackage{amsfonts}
\newtheorem{thm}{Theorem}[section]

\newtheorem{propo}[thm]{Proposition}
\newtheorem{prop}[thm]{Properties}
\theoremstyle{definition}

\theoremstyle{remark}

\numberwithin{equation}{section}
\newtheorem{ex}[thm]{Example}

\usepackage[ruled]{algorithm2e} %[ruled,vlined, longend]
\usepackage{environ}
\makeatletter
\newsavebox{\measure@tikzpicture}
\NewEnviron{scaletikzpicturetowidth}[1]{%
  \def\tikz@width{#1}%
  \def\tikzscale{1}\begin{lrbox}{\measure@tikzpicture}%
  \BODY
  \end{lrbox}%
  \pgfmathparse{#1/\wd\measure@tikzpicture}%
  \edef\tikzscale{\pgfmathresult}%
  \BODY
}
\makeatother
\def\R{\mathbb{R}}
\def\Q{\mathbb{Q}}
\def\Z{\mathbb{Z}}
\def\B{\mathbb{B}}

\def\N{\mathcal{N}}
\def\D{\mathcal{D}}
\def\X{\mathcal{X}}
\newcommand{\dsum}{\displaystyle\sum}
\newcommand{\dmin}{\displaystyle\min}

\newcommand{\dprod}{\displaystyle\prod}

\usepackage{multirow}

\usepackage{tikz}
\usetikzlibrary{shapes,arrows}
\usepackage{pgfplots}
\pgfplotsset{compat=newest}
\usetikzlibrary{decorations.markings}

\usepackage{hyperref}

\let\origmaketitle\maketitle
\def\maketitle{
  \begingroup
  \def\uppercasenonmath##1{} % this disables uppercasing title
  \let\MakeUppercase\relax % this disables uppercasing authors
  \origmaketitle
  \endgroup
}

\usepackage{lscape}
\begin{document}

\title{Ordered $p$-median problems with neighbourhoods}%

\author[V. Blanco]{V\'ictor Blanco\\Dpt. Quant. Methods for Economics \& Business, Universidad de Granada\\vblanco@ugr.es}
\email{vblanco@ugr.es}

\address{Dpt. Quant. Methods for Economics \& Business, Universidad de Granada}

\date{\today}

\begin{abstract}
In this paper, we introduce a new variant of the $p$-median facility location problem in which it is assumed that the exact location of the potential facilities is unknown. Instead, each of the facilities must be located in a region around their initially assigned location (the neighborhood). In this problem, two main decisions have to be made simultaneously: the determination of the potential facilities that must be open to serve the demands of the customers and the location of the open facilities in their neighborhoods, at global minimum cost. We present several mixed integer non-linear programming formulations for a wide family of objective functions which are common in Location Analysis: ordered median functions. We also develop two math-heuristic approaches for solving the problem. We report the results of extensive computational experiments.
\end{abstract}


\maketitle

\section{Introduction}

Facility location concerns with the optimal placement of one or several new facilities/plants to satisfy the customers' demands. In discrete facility location problems, the positions of both the customers and the potential new facilities are part of the input as well as the travel costs between them. On the other hand, in continuous facility location problems, although the (geographical) coordinates (in a given $d$-dimensional space) of the customers are provided, the information about the potential location of the facilities is unknown, in the sense that the facilities can be located at any place of the given space. Both the discrete and the continuous versions of facility location problems have been widely studied in the literature (see the monographs \cite{DH_2002,NP_2005} and the references therein). Several versions of these facility location problems have been analyzed, by considering different objective functions \cite{RNPF_MMOR00}, by fixing either the number of facilities to be located (as in the $p$-median or $p$-center problems) \cite{H_OR64} or maximum capacities for the facilities (capacitated facility location)\cite{KH_MS63,P_ORP08}, or assuming uncertainty in the demands of the customers (see \cite{AFS_O11,CNS_O17,CS_2015} for a recent review), amongst many others. 

In this paper, we propose a unified framework for facility location problems in which the underlying problem is a discrete facility location problem. However, because of locational imprecision or unaccuracy, the new facilities are allowed to be located not only in the exact location of the potential facilities, but in certain regions around each of them, the \textit{neighborhoods}. In case the initial placements of the potential facilities are exact enough, that is, their neighborhoods are singletons (with a single element which coincides with the initial placement of the potential facilities), the problem becomes the discrete location version of the problem. On the other hand, if the neighborhoods are large enough, the problem turns into the continuous location version of the problem, allowing the facilities to be located in the entire space. Otherwise, different shapes and sizes for the neighborhoods allow one to model how imprecise the locational information provided is. The goal is, apart from the discrete location decision of the problem (placement  of facilities among the given set and allocations customers-plants), to find the optimal location of the open facilities in the neighborhoods. The main difference between this problem and its underlying discrete facility location problem, is that in the latest, the travel distances between facilities and customers are assumed to be known, while in the neighborhood version of the problem, as in the continuous case, those distances depend on the place where the facility is located in the neighborhood. Hence, in this problem, the matrix of travel costs is not provided, but a distance measure to compute the travel costs between customers and facilities is given. This problem, as far as we know, has not been fully  investigated in Location Analysis, although some attempts have been presented in \cite{Cooper_JRS78} and \cite{Juel_OR81} where sensitivity analyses were performed by allowing the customers to \textit{move} around disc-shaped neighborhoods on the plane. Also, this problem can be seen as a constrained version of the classical multifacility location problem, which have been only partially studied in the literature (see \cite{BPE_EJOR16}). This framework will be called \textit{Facility Location with Neighborhoods}, a terminology borrowed from the neighborhood versions of the Minimum Spanning Tree problem \cite{BFP_ARXIV16,DFHKKLS_TCS15} and Traveling Salesman problem  \cite{B_JA05,DM_JA03}. 

The importance of analyzing this family of problems comes from its wide range of applications. It is well known that discrete facility location problems are useful in many real-world applications (see \cite{LG_IJOC12,MNS_CORS06}, amongst many others). However, in many situations, as for instance in the design of telecommunication networks, where a set of servers must be located to supply connection to a set of customers, the exact location of a server may not be exactly provided. In contrast, a region where the decision maker wishes to locate each of the facilities (a corridor, a room, or any other bounded space) can be \textit{easily} given. In such a case, a robust worst-case decision would not reflect reality, since the decision maker does not known the location of the facility because a lack of certainty but because it allows locational flexibility to the decision. An optimal design may be obtained if the new facilities are allowed to be located in adequately chosen neighborhoods.

In this paper, we provide suitable mathematical programming formulations for the neighborhood versions of a widely studied family of objective functions in facility location problems: ordered median (OM) functions. In these problems, $p$ facilities are to be located by minimizing a flexible objective function that allows one to model different classical location problems. For instance, OM problems allow modeling location problems in which the customers support the median ($p$-median) or the maximum ($p$-center) travel costs, among many other robust alternatives. 
OM problems were introduced in Location Analysis by Puerto and Fern\'andez \cite{PF_JNCA00} and several papers have analyzed this family of objective functions in facility location: discrete problems \cite{KNPR_TOP10,LPP_COR17,Ponce2016}, continuous problems \cite{BEP_COR13,BEP14}, network/tree location problems~\cite{KNP_NETW03,PT_MP05,Tang_etal16}, hub location problems~\cite{PRR_COR11}, stochastic facility location problems~\cite{Yan_etal17}, multiobjecive location\cite{Grabis_etal12}, etc  (see  \cite{PR_chapter2015} for a recent overview on the recent developments on ordered median location problems). In particular, we analyze the neighborhood version of OM location problems for the so-called monotone case. We study the still general case in which the neighborhoods are second-order cone representable regions. These sets allow one to model as particular cases polyhedral neighborhoods or $\ell_\tau$-norm balls. The distance measure to represent travel costs between customers and facilities are assumed to be $\ell_\nu$-norm based distances. Within this framework we present four different Mixed Integer Second Order Cone Optimization (MISOCO) models.

The current limitations of the on-the-shelf solvers to solve mixed integer nonlinear problems, and the difficulty of solving even the underlying problem (the classical $p$-median problem is NP-hard), makes the resolution of the problem under study a hard challenge. For that reason, we also develop two math-heuristic algorithms based on different location-allocation schemes, which are able to solve larger problems.

Our paper is organized in five sections. In Section \ref{sec:1} we introduce the problem and some general properties are stated. Section \ref{sec:formulations} is devoted to provide four different mixed integer non linear programming formulations of the problem. At the end of the section, we run some computational experiments in order to compare the four formulations. In Section \ref{heuristics} the two math-heuristic approaches are described, and the results of some computational experiments are reported. Finally, some conclusions are presented in Section \ref{sec:conclusions}.

\section{DOMP with Neighborhoods}
\label{sec:1}

In this section we introduce the Ordered Median Problem with Neighborhoods (OMPN) in which the underlying discrete facility location problem is the Discrete Ordered $p$-Median Problem (DOMP).

For the sake of presentation, we first describe the DOMP problem. The input data for the problem is:
\begin{itemize}
\item  $\mathcal{A}=\{a_1, \ldots, a_n\} \subseteq \R^d$: set of coordinates of the customers. We assume, as usual in the location literature, that the coordinates of potential facilities coincides with $\mathcal{A}$.
\item $D = \Big(d(a_i,a_j)\Big)_{i,j=1}^n \in \R^{n\times n}$: Travel cost matrix between facilities.
\item $\lambda_1, \ldots, \lambda_n \geq 0$: Ordered median function weights.
%\item $\{f(a) \in \R_+: a\in\mathcal{A}\}$: Set-up costs for the facilities.
\end{itemize}
 
The goal of DOMP is to select, from the elements of $\mathcal{A}$, a subset of $p$ facilities, $\mathcal{B} \subset \mathcal{A}$ with $|\mathcal{B}|=p$, that minimizes the ordered median objective function:
$$
\dsum_{i=1}^n \lambda_i D_{(i)},
$$
where $D_i= \min_{b\in \mathcal{B}} d(a_i,b)$ (the smallest travel cost to supply customer $i$ from the open facilities), and $D_{(i)}$ represent the $i$-th largest element in the set $\{D_1, \ldots, D_n\}$, i.e. $D_{(i)} \in \{D_1, \ldots, D_n\}$ with $D_{(1)} \geq \cdots \geq D_{(n)}$.  

The DOMP can be stated as the following optimization problem:
\begin{equation}
\min_{\mathcal{B} \subset \mathcal{A}: |\mathcal{B}|=p} \dsum_{i=1}^n \lambda_i D_{(i)}% + \dsum_{b\in\mathcal{B}} f(b)
\label{domp0}\tag{${\rm DOMP}$}
\end{equation}

We will assume that the $\lambda$-weights verify $\lambda_1 \geq \cdots \geq \lambda_n \geq 0$, dealing with the so-called \textit{convex ordered median problem}. Most of the main well-known objective functions in Locational Analysis are part of this family, as for instance:
\begin{itemize}
\item Median ($\lambda=(1,\ldots, 1)$):   $\sum_{i=1}^{n} D_i$.
\item Center ($\lambda=(1,0,\ldots,0)$: $\max_{i=1, \ldots, n} \;D_i$.
\item $K$--Centrum $\lambda=(1,\stackrel{K}{\ldots},1,0,\ldots,0)$: $\sum_{i=1}^{K} D_{(i)}$.
\item Cent-Dian$_{\alpha}$ ($\lambda=(1,1-\alpha, \cdots, 1-\alpha)$):  $\alpha \max_{i=1, \ldots, n} D_{i}+(1-\alpha)\sum_{1\leq i\leq n} D_{i}$, for $0\leq\alpha\leq1$.
\end{itemize}

Ordered median functions are continuous and symmetric (in the sense that they are invariant under permutations). Furthermore, if $\lambda_1 \geq \ldots \geq \lambda_n \geq 0$, ordered median functions are convex, fact that will be exploited throughout this paper. The interested reader is referred to \cite{PR_chapter2015} for a complete description of the properties of ordered median functions. 

A few formulations and exact solution approaches for DOMP have been developed since the problem was introduced. In particular Boland et. al \cite{BDNP_COR06} formulated the problem as a (non convex) quadratic problem with quadratic constraints. A suitable three index (pure) binary programming reformulation with $O(n^3$) variables and $O(n^2)$ linear constraints was provided by linearizing the bilinear terms. A second formulation, reducing to two the indices of the variables in the formulation, was also presented in the same paper by using a different linearization strategy, and that allows reducing the number of binary variables to $O(n^2)$. Puerto \cite{P_ORP08}, Marin et. al \cite{MNV_MMOR10}, Marin et. al \cite{MNPV_DAM09} and Labb\'e et. al \cite{LPP_COR17} provided alternative formulations for the problem with two and three indices, that need, in a preprocessing phase, sorting the elements in the matrix $D$ (and removing duplicates).  All the above mentioned formulations are valid for general ordered median problems. Concerning the convex case, Ogryzack and Tamir \cite{OT_IPL03} presented a different formulation which exploits the monotonicity of the $\lambda$-weights by applying a $k$-sum representation of the ordered median function (see also the recent paper \cite{PRT_MP16} for further details on the powerful of this representation in a wide variety of optimization problems). Finally, in Blanco et. al  \cite{BEP14} the authors derived a formulation that also avoid using the binary variables for sorting the involved distances. Also, a few heuristic approaches are available in the literature for the DOMP problem (see \cite{DNHM_AOR05,PPG_EJOR14,SKD_EJOR07}).

Observe also that in the DOMP, once the travel costs matrix is provided, the locational coordinates of the customers are not needed, and then, the problem does not depend on  the dimension of the space where the customers live.

For the OMPN framework, instead of providing a travel cost matrix between customers, we consider a travel distance measure $d: \R^d\times \R^d \rightarrow \R_+$ induced by a norm $\|\cdot\|$, i.e., $d(a,b) = \|a-b\|$, for $a, b \in \mathcal{A}$.

Also, each potential facility, $a\in \mathcal{A}$, is associated to a convex set, $\N(a) \subset \R^d$, with $a \in \N(a)$, its \textit{neighbourhood}. We denote by $\overline{\N} = \dprod_{a\in\mathcal{A}} \N(a)$, the space of neighborhoods. We also consider in this case, set-up costs for opening facilities (which may be neighborhood-dependent), which we denote by  $f(a)$ for each $a\in \mathcal{A}$.

The goals of the Ordered Median Problem with Neighborhoods are:

\begin{itemize}
\item to find the indices of the $p$ facilities to open: $\mathcal{B}=\{b_1, \ldots, b_p\}$,with $b_j \in \mathcal{A}$ for $j=1, \ldots, p$,
\item to locate the facilities into their neighbourhoods: $\bar b_1, \ldots, \bar b_p$ with $\bar b_j \in \N(b_j)$, $j=1, \ldots, p$, and
\item to allocate customers to their closest open facilities $\bar b_1, \ldots, \bar b_p$,
\end{itemize}
by minimizing an ordered median function of the travel distances plus set-up costs.

Observe that the optimization problem to solve for the OMPN is similar to \eqref{domp0}:

\begin{equation}
\min_{\stackrel{\mathcal{B} \subset \mathcal{A}: |\B|=p}{\bar a \in \N}} C(\mathcal{B}) := \dsum_{i=1}^n \lambda_i D_{(i)} + \dsum_{b \in \mathcal{B}} f(b)\label{dompn0}\tag{${\rm OMPN}$}
\end{equation}
but now, $D_i = \min_{b \in \mathcal{B}} d(a_i, \bar b)$, i.e. the travel distance from a customer to its closest facility depends on the position of the facilities in their neighborhoods. So both the discrete location (open facilities and allocation scheme) and the continuous location decisions (coordinates of the new facilities) are involved in the problem.

We use the classical notation for the variables in $p$-median problem:
$$
x_{ij} = \left\{\begin{array}{cl}
1 & \mbox{if client $i$ is allocated to facility $j$ ($i\neq j$) or if facility $j$ is open ($i=j$),}\\
0 & \mbox{otherwise}
\end{array}\right.
$$
for $i, j=1, \ldots, n$.

Note that, using the above family of variables, the set of open facilities and assignments between customers and $p$ facilitites can be represented by the set $\X = \X_R \cap \{0,1\}^{n\times n}$, where 
\begin{eqnarray*}
\X_R = \Big\{x \in  [0,1]^{n\times n}:  \dsum_{j=1}^n x_{ij} =1,  \forall i=1, \ldots, n, \dsum_{j=1}^n x_{jj}=p, x_{ij} \leq x_{jj}, \forall i, j=1, \ldots, n\Big\}
\end{eqnarray*}
is the so-called \emph{$p$-median polytope}.

Observe also that, the above settings easily extend to the case in which the possible connections between demand points and facilities is induced by a graph.

On the other hand, the set of distances, will be represented by the following set:
\begin{eqnarray*}
\D = \Big\{(d,\bar a) \in \R^{n\times n}_+ \times \overline \N: d_{ij} \geq \|a_i - \bar a_j\|, i, j=1, \ldots, n, i\neq j\Big\},
\end{eqnarray*}
where $d_{ij}$ (when one tries to minimize some aggregating function of the travel-costs) represents the distance between the customer located at $a_i$ and the facility located at $\bar a_j$, for all $i, j=1, \ldots, n$.% Note that the goal of OMPN is to locate the facilities to provide a minimum \emph{ordered weigthed sum} of travel costs of the customers, so sorting the set-up costs of the facilities as part of the travel costs is not possible under this framewok.

Note that the set $\D$ can be \emph{easily} adapted to the case in which each customer uses a different travel distance measure (norm), and the structure of $\D$ remains the same.

With the above notation, the general OMPN can be compactly formulated as:
\begin{align}
\min &\dsum_{i=1}^n \lambda_i z_{(i)} + \dsum_{j=1}^n f_j x_{jj}\label{dompn:0}\\
\mbox{s.t. } & z_i = \dsum_{i=1}^n d_{ij} x_{ij}, i, j=1, \ldots, n, \label{bilinear}\\
& x \in \X, (d,\bar a) \in \D.\label{dompn:D}
\end{align}
where $f_j$ denotes the set-up cost of the facility initially located at $a_j$, $j=1, \ldots, n$ and $z_i$ represents the minimum distance between customer located at $a_i$ and the open facilities.

Observe that \eqref{dompn:0}--\eqref{dompn:D} is a mixed integer non linear programming problem (MINLP), whose continuous relaxation is not convex nor concave due to the bilinear constraint \eqref{bilinear} and probably to the constraints in $\D$. In case the neighborhoods are convex, the set $\D$ is also convex (because of the convexity of the norm). Hence, if the discrete location variables $x$ were known, the problem (also because the convexity of the ordered median function) becomes a continuous convex problem. On the other hand, if the distances were known, the problem becomes a DOMP, so several formulations can be applied to solve the problem. In the OMPN, both $\X$ and $\D$ are part of the final decision. Thus, both the difficulties of handling the DOMP problem and the continuous problem are inherited to the OMPN. In particular, since the $p$-median problem (or the $p$-center problem) is known to be NP-hard \cite{KarivHakimi_SIAMAM79} which is a particular case of OMPN, the OMPN is also NP-hard.

The simplest OMPN problem, apart from the DOMP case (where the neighborhoods can be seen as singletons), is obtained when the set $\D$ is a polyhedron (and then, defined by a set of linear inequalities). Since the geometry of $\D$ depends on the distance measure induced by $\|\cdot\|$ and the shapes of the neighborhoods, $\D$ will be a polyhedron when these two features can be linearly represented. The norms which are polyhedrally-representable are called \textit{block} (or polyhedral) norms (see \cite{NP_2005}) which are characterized by the fact that their unit balls are polytopes, i.e., $P=\{z\in \R^d: \|z\|\leq 1\}$ is a bounded polyhedron. On the other hand, the neighborhoods, because they are assumed to be compact and convex sets, their polyhedral representability is assured if and only if they are also polytopes. In those cases, both the set $\D$ and $\X$ are identified with sets of linear inequalities (and integrality constraints in $\X$). Furthermore, as can be checked in \cite{OO2012} or \cite{PR_chapter2015}, the ordered median function can be also modeled by using a set of linear inequalities and equations and by adding a set of $O(n^2)$ binary variables to our model. The above observations are summarized in the following result.
\begin{thm}
Let $\lambda_1 \geq \cdots \geq \lambda_n \geq 0$, $\|\cdot\|$ a block norm and $\N(a)$ a polyhedron, for each $a \in \mathcal{A}$. Then OMPN can be formulated as a
mixed-integer linear programming problem.
\end{thm}
\begin{proof}
The proof follows noting that constraints in the form $Z \geq \|X-Y\|$, as those that appear in the description of $\D$, can reformulated as:
$$
Z \geq e^t (X-Y), \; \forall e \in {\rm Ext}(P^*),
$$
where ${\rm Ext}(P^*)$ the set of extreme points of $P^* = \{v \in \R^d: v^t b_g  \leq 1, g=1, \ldots, |Ext(P)|\}$, the unit ball of the dual norm of $P$ (see \cite{NP_2005,Ward-Wendell} ).
\end{proof}

The following example illustrates the new framework under study.

\begin{ex}\label{ex:1}
Let us consider a set of customers/potential facilities with coordinates in the plane $\mathcal{A}=\{(0,5)$, $(1,1)$, $(1,6)$, $(1,4)$, $(5,3)$, $(10,4)$, $(6.5,0)$, $(8,6)\}$, and travel distances measured with the Euclidean norm. The solutions for the $2$-median, the $2$-center and the $2$-$4$-center ($K$-center with $K=4$) are drawn in Figure \ref{fig:ex1_0} (there, stars represent open facilities, customers are identified with dots and lines are the allocation patterns).

\begin{figure}[h]
\input{ex0-1.tex}
\caption{Solutions for $2$-median, $2$-center and $2$-$4$-center for the data in Example \ref{ex:1}.\label{fig:ex1_0}}
\end{figure}
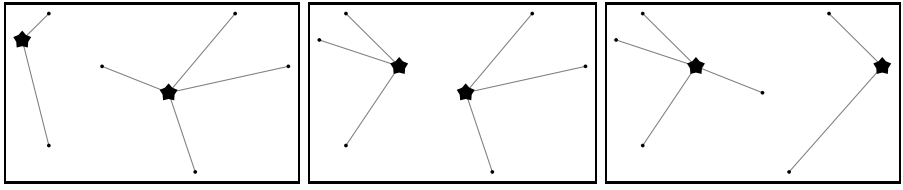

Note that, as expected, the solutions for the three DOMP problems highly depend on the $\lambda$-weights,  being the optimal set of open facilities different for the three cases.

Let us now consider, for each demand point, a neighbourhood defined as the Euclidean disk with radii $r \in \{1, 0.6, 1, 0.6, 2.4, 2.4, 0.8, 1.6\}$ (see Figure \ref{fig:ex1_1}).

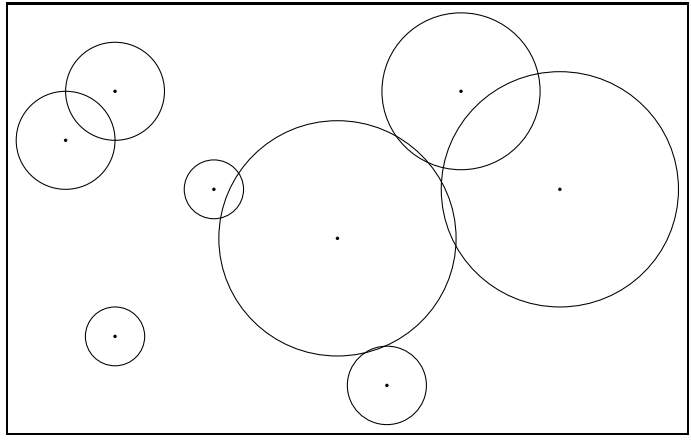
\begin{figure}[h]
\input{ex0-2.tex}
\caption{Neighbourhoods for the facilities of Example \ref{ex:1}.\label{fig:ex1_1}}
\end{figure}

The new facilities, now, are not restricted to be exactly located in the given coordinates but in a disk around them. We consider the radius of its neighborhood (disk) as a mesure of the set-up cost for  each facility, that is $f(a_i)=r_i$. The solutions of the neighbourhood version of the $2$-median, $2$-center and $2$-$4$-center problems are shown in Figure \ref{fig_ex1_2}.

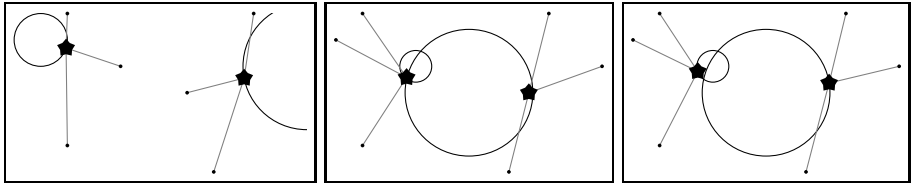
\begin{figure}[h]
\input{ex0-3.tex}
\caption{Solutions for $2$-median, $2$-center and $2$-$4$-center with neighbourhood for the data in Example \ref{ex:1}.\label{fig_ex1_2}}
\end{figure}

%Note that when neighbourhoods around facilities are considered, both the open facilities and the allocation of customers may change with respect to the DOMP solutions.

\end{ex}

In what follows, we derive some structural properties of DOMP that are inherited to the OMPN.

For each $i,j =1, \ldots, n$, we denote by $\widehat{D}_{ij} = \max \{\|a_i - \bar a_j\|: \bar a_j \in \N_j\},$ and $\widehat{d}_{ij} = \max \{\|a_i - \bar a_j\|: \bar a_j \in \N_j\}$, upper and lower bounds for the distances between the $i$th customer and the $j$th potential facility, respectively.

\begin{prop}\label{prop:1}
The following properties are satisfied:
\begin{enumerate}
\item There exists an optimal solution of \eqref{dompn0} in which the $p$ smaller travel distances equal $0$.
\item Let $\mathcal{B}\subseteq \mathcal{A}$ a set of $p$ facilities such that its ordered cost $C(\mathcal{B}) \leq UB$ and such that $\min_{j\neq i} \widehat{d}_{ij} > \dfrac{UB}{\dsum_{i=1}^m \lambda_i}$ for some $m=2, \ldots, n$, then the $i$-th client is sorted at most in position $m$ in the whole sorted set of optimal distances.
\end{enumerate}
\end{prop}

\begin{proof} $ $

{\it 1.} The result follows from the observation that if the facility $a \in \mathcal{A}$ is an open facility, the travel costs between $a$ and $a$ are zero.

{\it 2.} Assume that the $i$th customer is sorted in position $r \geq m$ in the sorting sequence of distances, i.e.,  $D_{(1)} \geq \ldots \geq D_{(m)} \geq D_{(r)} = D_i$. Then, we have that:
$$
c(\mathcal{B}) = \dsum_{l=1}^n \lambda_l D_{(l)} + \dsum_{b\in B} f(b) \geq \dsum_{l=1}^m \lambda_l D_{(l)} \geq \dsum_{l=1}^m \lambda_l D_{i} = D_i \left(\dsum_{l=1}^m \lambda_l\right) > UB
$$
which contradicts the hyphotesis.
\end{proof}

%The above properties allows one fixing to zero some of the binary variables that appear in some formulations of the OMPN, to solve the problem. 

\section{MINLP Formulations for the OMPN}
\label{sec:formulations}

In this section, we describe different mathematical programming formulations for solving general OMPN. In particular, we extend the formulations presented in \cite{BEP14}, \cite{BDNP_COR06} and \cite{OT_IPL03} to our problem. As mentioned above, the main difference between the DOMP and the OMPN problem is that in the OMPN the distances are not part of the input, but part of the decision. Hence, the formulations for the DOMP  based on preprocessing the travel distances matrix (as those proposed in \cite{LPP_COR17},\cite{MNV_MMOR10}, \cite{MNPV_DAM09} or \cite{P_ORP08}) cannot be applied to our framework.

Observe that, in OMPN, an adequate representation of $\D$ is crucial for the development of efficient solution approaches for the problem. We assume that the neighborhoods belong to a family of convex sets that allows us to represent most of the convex shapes which are useful in practice, and that can be efficiently handled by commercial optimization solvers: second order cone (SOC)-representable sets \cite{LVBL_SOC98}. SOC-representable sets are convex sets defined by second-order cone constraints in the form:
\begin{equation}\label{soc}\tag{SOC}
\|A_i\,x-b_i\|_2 \leq c_i^t x + d_i, \forall i=1, \ldots, M,\\
x \in \R^{N},
\end{equation}
where $A_1, \ldots, A_i \in \R^{M_i\times N}$, $b_i \in \R^{M_i}$, $c_i\in \R^{N}$, $d_i \in \R$, for $i=1, \ldots, M$, and $\|\cdot\|_2$ is the Euclidean norm. Most of the state-of-the-art solvers are capable to efficiently solve optimization problems involving SOC constraints by means of quadratic constraints with positive definite matrices, second order cone constraints (in the form $x^t x \leq y^2$, for $y\geq 0$) or rotated second order cone constraints ($x^t x \leq yz$ with $y, z\geq 0$). SOC constraints allow one to represent, not only Euclidean balls, but any $\ell_\tau$-norm ball (see \cite{BEP14} for further details on the explicit representation of $\ell_\tau$-norm based distance constraints as a set of SOC constraints for any $\tau\in \Q$ with $\tau\geq 1$). Clearly, any polyhedron is SOC-representable (setting $A$ and $b$ equal to zero) so any intersection of $\ell_\tau$-norm balls and polyhedra is suitable to be represented as a set of second order cone constraints. Hence, both our neighborhoods and the distances involved in our problem will be defined as SOC-constraints, being then $\D$ a SOC-representable set.

For the sake of simplicity, and without loss of generality, we assume that the neighborhood of each $a\in \mathcal{A}$ is a $\ell_\tau$-norm ball, i.e. $\N(a)= \{z \in \R^d: \|z-a\|_\tau \leq r_a\}$, for some $r_a \in \R_+$ and $\tau \in \Q$ with $\tau \geq 1$.

Also, we consider that the travel distances are induced by a $\ell_\nu$-norm with $\nu\in \Q$ and $\nu\geq 1$. With these settings, we  explicitly describe $\D$ as follows:
$$
\D = \{(d,\bar a) \in \R_+^{n\times n} \times \R^{n\times d}: d_{ij} \geq \|a_i - \bar a_j\|_\nu,  r_j  \geq \|a_j - \bar a_j\|_\tau, i,j=1, \ldots, n\}
$$
where $r_j$ denotes the radius of the neighborhood $\N(a_j)$, i.e., $r_j=r_{a_j}$.
 
 The following result, whose proof is straightforward from \cite[Theorem 2]{BEP14}, allows us to efficiently represent the set $\D$ when the involved norms are $\ell_\tau$-based norms.
 \begin{prop}
 Let $\tau= \frac{r_\tau}{s_\tau}\geq 1$ and $\nu=\frac{r_\nu}{s_\nu} \geq 1$ with $r_\tau, s_\tau, r_\nu, s_\nu \in \Z_+$ and $\gcd(r_\tau,s_\tau)=\gcd(r_\nu,s_\nu)=1$. Then, $\D$ is representable as a set of $(n^2 +n)(2d+1)$ linear inequalities and $nd(n\log\;r_\nu  + \log\;r_\tau)$ second order cone contraints.
 \end{prop} 

\subsection{The three index formulation}

The first formulation is based on the one proposed in \cite{BDNP_COR06}, which uses, apart from the $x_{jj}$-variables described above, the following set of sorting/allocation binary variables for the DOMP:

$$
w_{ij}^k = \left\{\begin{array}{cl} 1 & \mbox{if customer $i$ is allocated to facility $j$ and its distance, $\|a_i-\bar a_j\|$,}\\
& \mbox{is sorted in the $k$th position.}\\
0 & \mbox{otherwise}.\end{array}\right.
$$

This formulation reads as follows:

\begin{align}
\min &\dsum_{i, j, k=1}^n \lambda_k d_{ij} w_{ij}^{k} + \dsum_{j=1}^n f_j x_{jj}\label{domp1}\tag{${\rm OMPN}_{3I}$}\\
\mbox{s.t. } & \dsum_{j,k=1}^n w_{ij}^{k}=1, \forall i=1, \ldots,n,\label{domp1:1}\\
&  \dsum_{i,j=1}^n w_{ij}^{k}=1, \forall k=1, \ldots,n,\label{domp1:2}\\
& \dsum_{k=1}^n w_{ij}^k \leq x_{jj}, \forall i, j=1, \ldots, n,\label{domp1:3}\\
& \dsum_{j=1}^n x_{jj}= p,\label{domp1:4}\\
& \dsum_{i,j=1}^n d_{ij} w_{ij}^{k-1} \geq \dsum_{i,j=1}^n d_{ij} w_{ij}^{k}, \forall k=2, \ldots, n,\label{domp1:5}\\
& w_{ij}^{k} \in \{0,1\}, \forall i,j,k,=1, \ldots, n,\nonumber\\
& x_{jj} \in \{0,1\}, \forall j=1, \ldots, n.\nonumber\\
& (d,\bar a) \in \mathcal{D}.\nonumber
\end{align}
The objective function assigns to each sorted distance its adequate weight $\lambda$. \eqref{domp1:1} (resp. \eqref{domp1:2}) ensures that each demand point (resp. each position) is assigned to a unique facility and a unique position (resp. demand point). \eqref{domp1:3} assures that allocation is not allowed unless the plant is open, and \eqref{domp1:4} restrict the problem to open exactly $p$ facilities. Constraints \eqref{domp1:5} allows us a correct definition of the $w$-variables in which the sorting of the distances is imposed (the $(k-1)$th is at least as larger as the $k$-th distance).

Although the above formulation is valid for  OMPN, both the objective function and the set of constraints \eqref{domp1:5} are quadratic and non-convex. We introduce a new set of variables to account for the non linear terms in the above formulation:
$$
\theta_{ij}^{k} = d_{ij} w_{ij}^k, \quad i, j, k=1, \ldots, n.
$$
Using the $\theta$-variables, the objective function can be reformulated as:
$$
\dsum_{i, j, k=1}^n \lambda_k \theta_{ij}^{k} + \dsum_{j=1}^n f_j x_{jj}.
$$
The correct definition of the new variables and satisfaction of constraints \eqref{domp1:5} is assured by the following sets of linear constraints:
\begin{align*}
\theta_{ij}^{k} \geq& d_{ij} - \widehat{D}_{ij}(1-w_{ij}^{k}), &\forall i, j, k=1, \ldots, n,\\
\dsum_{i, j=1}^n \theta_{ij}^{k-1} \geq& \dsum_{i, j=1}^n  \theta_{ij}^{k}, &\forall i, j, k=1, \ldots, n,
\end{align*}
where the first set of constraints comes from the McCormick linear reformulation \cite{M_MP76} of the bilinear terms defining the $\theta$-variables, and the second is the reformulation of \eqref{domp1:5} with the new variables.

The formulation above, has $O(n^3)$ variables and $O(n^2)$ constraints. Properties \ref{prop:1} allow us to strengthen the formulation \eqref{domp1}. In particular, if $UB$ is a known upper bound for the optimal value of OMPN, then:
$$
w_{ij}^k = 0, \forall i, j, k =1, \ldots, n, \mbox{ such that } \dmin_{j\neq i} \widehat{D}_{ij} > \dfrac{UB}{\dsum_{l=k}^n \lambda_l}.
$$
Also, because of the relationship between the $w$ and $x$ variables we get that
$$
\dsum_{k=1}^n w_{jj}^k = x_{jj}, \forall j=1,\ldots, n,
$$
are valid equations for \eqref{domp1}.

\subsection{The $2$-index formulation}

The second formulation, also based on the one presented in \cite{BDNP_COR06}, considers an alternative representation of the sorting variables. It uses two different sets of variables. The first one allows us to sort the distances of supplying each of the customers:
$$
s_{ik}= \left\{\begin{array}{cl} 1 & \mbox{if the distance supported by the $i$th customer is sorted in the $k$th position.}\\
0 & \mbox{otherwise}.\end{array}\right.
$$
while the second represents the sorting (non decreasing) sequence of distances:
$$
\xi_k = \dsum_{i=1}^n s_{ik} \dsum_{j=1}^n d_{ij} x_{ij}, \quad k=1, \ldots, n.
$$
This representation allows us to simplify the formulation to the following with $O(n^2)$ variables and $O(n^2)$ constraints.

\begin{align}
\min & \dsum_{k=1}^n \lambda_k \xi_k + \dsum_{j=1}^n f_j x_{jj}\label{domp2}\tag{${\rm OMPN}_{2I}$}\\
\mbox{s.t. } & \xi_{k} \geq \xi_{k+1}, \forall k=1,\ldots, n-1,\label{domp2:1}\\
& \dsum_{k=1}^n \xi_k = \dsum_{i,j=1}^n d_{ij} x_{ij},\label{domp2:2}\\
& \xi_k \geq  d_{ij}x_{ij} - \widehat{D}_{ij}\; \left(1- s_{ik}\right), \forall i, k=1, \ldots, n,\label{domp2:3}
\end{align}
\begin{align}
&& \dsum_{i=1}^n s_{ik}= 1,\forall k=1, \ldots, n\label{domp2:4}\\
& \dsum_{k=1}^n s_{ik}= 1,\forall i=1, \ldots, n,\label{domp2:5}\\
& \xi_{i}\geq 0, \forall i, k=1, \ldots, n, \nonumber\\
& s_{ik} \in \{0,1\}, \forall i, k=1, \ldots, n,\nonumber\\
& x \in \X,  (d,\bar a) \in \mathcal{D}.\nonumber
\end{align}

The correct definition of the $\xi$-variables is assured by constraints \eqref{domp2:1}--\eqref{domp2:3}, while contraints \eqref{domp2:4} and \eqref{domp2:5} allows us the adequate modeling of the $s$-variables.

As in \eqref{domp1}, to avoid nonconvex terms in the formulation, the bilinear terms $d_{ij}x_{ij}$ can be linearized by introducing a new variable $\theta_{ij} = d_{ij}x_{ij}$ and replacing \eqref{domp2:2} and \eqref{domp2:3} by:
\begin{align}
\dsum_{k=1}^n \xi_k &= \dsum_{i,j=1}^n \theta_{ij},\label{domp2:6}\\
\xi_k & \geq \theta_{ij}  - \widehat{D}_{ij} \left(1- s_{ik}\right), \forall i, j, k=1, \ldots, n.\label{domp2:7}\\
\theta_{ij} & \geq d_{ij}  - \widehat{D}_{ij} \left(1-x_{ij}\right), \forall i, j, k=1, \ldots, n.\label{domp2:8}
\end{align}

\subsection{The $K$-sum formulation}
\label{formulation:OT}

Ogryczak and Tamir  presented  in \cite{OT_IPL03} some linear programming formulations for the problem of minimizing the sum of the $K$ largest linear functions (which is a particular case of ordered median function). In the same paper, the approach is extended to the minimization of convex ordered median functions by means of a telescopic sum of $K$-sum functions. In the next formulation, we apply this idea to formulate the OMPN. For the sake of readability, we first formulate the $K$-center problem.

Let $\lambda=(1, \stackrel{K)}{\ldots}, 1, 0, \stackrel{n-K)}{\ldots}, 0)$. The ordered median function associated to this particular choice of $\lambda$ is known as the $K$-center function. With our notation, provided the set of distances $D_1, \ldots, D_n$, the $K$-center problem consists of minimizing $\Theta_K(D) = \sum_{i=1}^K D_{(i)}$. Such an objective function, is proved in \cite{OT_IPL03} to be equivalent to the following expression
$$
\Theta_K(D) = \dfrac{1}{n} \left( K \dsum_{i=1}^n D_i + \min_{t\in \R} \dsum_{i=1}^n (K\;(t-D_i)_+ + (n-K)\;(D_i-t)_+)\right)
$$
where $z_+=\max \{0, z\}$ for $z\in \R$, and the optimal value $t^*$ into the above expression coincides with $D_{(K)}$ (the $K$-th largest distance). Hence, to minimize $\Theta_K(D)$ one may proceed by solving:
\begin{align*}
\min \;\;K\;t + \dsum_{i=1}^n z_i\\
\mbox{s.t. } & z_i \geq D_i-t, \forall i=1, \ldots, n,\\
& z_i \geq 0, \forall i=1, \ldots, n,\\
& t \in \R.
\end{align*}
where the variable $z_i$ is identified with $(D_i -t)_+$ in the above formulation, for $i=1, \ldots, n$. Thus, incorporating the whole information to define the distances in our location problem, the $K$-center location problem with neighborhoods can be formulated as:
\begin{align}
\min & \;\; K\; t + \dsum_{i=1}^n z_i  + \dsum_{j=1}^n f_j x_{jj}\label{KC}\tag{${\rm KCN}_{OT}$}\\
\mbox{s.t. } & z_i \geq D_i-t, \forall i=1, \ldots, n,\\
& D_i \geq d_{ij}- \widehat{D}_{ij} (1-x_{ij}), \forall i,j, =1, \ldots, n,\label{kc:1}\\
& z_i, D_i \geq 0, \forall i=1, \ldots, n,\\
& t \in \R,\\
& x \in \X, (d,\bar a) \in \mathcal{D}.\nonumber
\end{align}

The above formulation can be extended to general convex ordered median functions. Observe that if $\lambda_1 \geq \cdots \geq \lambda_n \geq \lambda_{n+1}:=0$ one may represent the ordered median function of the distances $D_1,\ldots, D_n$ by using a telescopic sum:
$$
\dsum_{i=1}^{n-1} \lambda_i D_{(i)} =\dsum_{k=1}^{n-1} (\lambda_k-\lambda_{k+1}) \dsum_{i=1}^K D_{(i)} =\dsum_{K=1}^n \Delta_k \Theta_K(D)
$$
where $\Delta_k=\lambda_k-\lambda_{k+1} \geq 0$ for $k=1, \ldots, n-1$.

Thus, the convex ordered objective functions can be equivalently rewritten as a weighted sum of $K$-sums, being then suitable to be represented as in the $K$-center problem. With such an observation and introducing new $t$-variables (one for each of the $K$-sums involved) and $z$-variables, in this case with two indices to account not only for the customer ($i$) but also for the $K$-sum representation, one obtain the following valid formulation for the OMPN:

\begin{align}
\min &\;\dsum_{k=1}^{n} \Delta_k (kt_k + \dsum_{i=1}^n z_{ik}) +\dsum_{j=1}^n f_j x_{jj}\ \label{domp3}\tag{${\rm OMPN}_{OT}$}\\
\mbox{s.t. } &z_{ik} \geq D_{i} - t_k,\forall i, k=1, \ldots, n,\nonumber\\
& D_i \geq d_{ij}- \widehat{D}_{ij} (1-x_{ij}), \forall i,j, =1, \ldots, n,\\
& z_{ik}, D_i \geq 0, \forall i=1, \ldots, n,\\
& t_k \in \R, k=1, \ldots, n,\\
& x \in \X, (d,\bar a) \in \mathcal{D}.\nonumber
\end{align}

Observe that this formulation has also $O(n^2)$ variables and $O(n^2)$ constraints, but, as will be shown in the computational results, it has a better performance than \eqref{domp2} since it uses, intrinsically, the especial structure of the convex ordered median objective. 

\subsection{The BEP formulation}

Finally, we present a formulation, based on the one provided in \cite{BEP14} and that, as the one in the previous subsection, is only valid for the convex case. The idea under the formulation comes from the observation that, because $\lambda_1 \geq \cdots \geq \lambda_n \geq 0$, the evaluation of the ordered median function on a set of distances, is reached when choosing, among all possible permutations of the indices, $\mathcal{P}_n$, the one that maximizes the weighted sum, that is:
$$
\dsum_{i=1}^n \lambda_i D_{(i)} = \max_{\sigma \in \mathcal{P}_n} \dsum_{i=1}^n \lambda_i D_{\sigma(i)}.
$$

Then, if the permutations of $\{1, \ldots, n\}$ are represented by using the set of binary variables 
$$
p_{ik}=\left\{\begin{array}{cl} 1 & \mbox{if the permutation assigns index $i$ ito index $k$},\\
0 & \mbox{otherwise}\end{array}\right.,
$$
verifying that $\sum_{i=1}^n p_{ik}=1$ (for all $k=1, \ldots, n$) and $\sum_{k=1}^n p_{ik}=1$ (for all $i=1, \ldots, n$). 

Then, using these variables, the ordered median sum of a given set of values $D_1, \ldots, D_n$ is equivalent to:
\begin{eqnarray*}
 \dsum_{i=1}^n \lambda_i D_{(i)} & = & \max_{p \in \{0,1\}^{n\times n}} \dsum_{i,k=1}^n \lambda_k D_i p_{ik}\\
 & & \mbox{s.t. } \dsum_{i=1}^n p_{ik}= 1, \forall k=1, \ldots, n,\\
  & & \qquad \dsum_{k=1}^n p_{ik}= 1, \forall i=1, \ldots, n.
  \end{eqnarray*}
The optimization problem above is an assignment problem, then, the total unimodularity of the constraints matrix assures that its optimal value coincides with the one of its dual problem which reads:
\begin{align*}
\min\;\; &\dsum_{k=1}^n u_k + \dsum_{i=1}^n v_i \\
\mbox{s.t. } & u_i + v_k \geq \lambda_k D_i, \forall i, k=1, \ldots, n,\\
& u, v \in \R^n.
\end{align*}

Merging the above representation of the ordered median function into the location problem, the OMPN is reformulated as: 

\begin{align}
\min & \dsum_{k=1}^n u_k + \dsum_{i=1}^n v_i + \dsum_{j=1}^n f_j x_{jj}\label{domp4}\tag{${\rm OMPN}_{BEP}$}\\
\mbox{s.t. } & u_i + v_k \geq \lambda_k D_i, \forall i,k=1, \ldots, n,\label{domp4:0}\\
& D_i \geq d_{ij} - \widehat{D}_{ij}(1-x_{ij}), \forall i, j=1, \ldots, n,\label{domp4:1}\\
& x \in \X,\nonumber\\
& (d,\bar a) \in \mathcal{D}.\nonumber
\end{align}

%where again \eqref{domp4:1} is the McCormick linearization of the bilinear term $D_i =\sum_{j=1}^n d_{ij}x_{ij}$. Valid inequalities as for the previous formulation are possible:
%$$
%D_i \leq \widehat{D}_{ij} x_{ij}, \forall i,j=1, \ldots, n
%$$

\subsection{Comparison of Formulations}
\label{sec:compform}

In this section we provide some theoretical results that allow us to compare the tightening of the convex relaxations of each of the provided formulations. Let us denote by $z^R_{3I}$, $z^R_{2I}$, , $z^R_{OT}$ and $z^R_{BEP}$  optimal values of the continuous relaxations of formulations \eqref{domp1}, \eqref{domp2}, \eqref{domp3} and \eqref{domp4}, respectively.

\begin{propo}
\label{relaxations}
The following relations are verified:
$$
z^R_{3I} \leq z^R_{2I} \leq z^R_{BEP} = z^R_{OT}.
$$
\end{propo}
\begin{proof}

Denote by $F_{3I}$, $F_{2I}$, $F_{OT}$ and $F_{BEP}$ the feasible regions of \eqref{domp1}, \eqref{domp2}, \eqref{domp3} and \eqref{domp4} obtained when relaxing the integrality conditions of the models.

\begin{itemize}
\item Let us consider the mapping $\pi:  \R^n_+\times \R^{n \times n}_+ \times [0,1]^{n\times n} \times \X_R \times \D \rightarrow \R^{n^3}_+ \times [0,1]^3 \times [0,1]^n \times \D $ defined as:
$$
\pi (\xi, \theta, s, x, (d, \bar a)) = ((\xi_k s_{ik}  x_{ij})_{i,j,k=1}^n, (s_{ik}x_{ij})_{i,j,k=1}^n, (x_{jj})_{j=1}^n, (d, \bar a))
$$

First, let us check that $\pi(F_{2I}) \subseteq F_{3I}$, which would prove the first inequality. Let $(\theta, \xi, s, x, (d, \bar a)) \in F_{2I}$, and define $(\bar \theta, \bar x, (d,\bar a))=\pi(\theta, \xi, s, x, (d, \bar a))$, i.e.:
$$
\bar \theta_{ij}^k = \xi_k s_{ik} x_{ij}, \; \bar w_{ij}^k = s_{ik}x_{ij}, \; \bar x_{jj}= x_{jj}, \; \forall i,j,k=1, \ldots, n.
$$
By construction, the constraints \eqref{domp1:1}-\eqref{domp1:4} are verified:
\begin{itemize}
\item $\dsum_{j, k=1}^n \bar w_{ij}^k =\dsum_{j, k=1}^n s_{ik} x_{ij} = \dsum_{j=1}^n x_{ij} \dsum_{k=1}^n s_{ik} \stackrel{\eqref{domp2:5}}{=} \dsum_{j=1}^n x_{ij} \stackrel{x\in \X_R}{=} 1$.
\item $\dsum_{i, j=1}^n \bar w_{ij}^k =\dsum_{i, j=1}^n s_{ik} x_{ij} = \dsum_{i=1}^n s_{ik} \dsum_{j=1}^n x_{ij} \stackrel{x\in \X_R}{=} \dsum_{i=1}^n s_{ik} \stackrel{\eqref{domp2:4}}{=} 1$.
\item $\dsum_{k=1}^n \bar w_{ij}^k =  \dsum_{k=1}^n s_{ik}x_{ij} =  x_{ij} \dsum_{k=1}^n s_{ik} \stackrel{\eqref{domp2:5}}{=} x_{ij} \stackrel{x\in \X_R}{=} x_{jj}$.
\item $\dsum_{j=1}^n \bar x_{jj} = \dsum_{j=1}^n x_{jj} = \stackrel{x\in \X_R}{=} p$.
\item $\bar \theta_{ij}^k = \xi_k s_{ik} x_{ij} \stackrel{\eqref{domp2:7}, \eqref{domp2:8}}{\geq}  (d_{ij}- \widehat{D}_{ij}(2-s_{ik}-x_{ij})) \, s_{ik}x_{ij} = d_{ij}- \widehat{D}_{ij}(1- s_{ik}x_{ij}) + (\widehat{D}_{ij} - d_{ij})(1-s_{ik}x_{ij}) + \widehat{D}_{ij}s_{ik} x_{ij} (s_{ik}+x_{ij}) \stackrel{\bar w_{ij}^k = s_{ik}x_{ij}, s_{ik}, x_{ij}\leq 1, d_{ij}\leq \widehat{D}_{ij}}{\geq} d_{ij}-  \widehat{D}_{ij}(1-\bar w_{ij}^k)$.
\item $\dsum_{i, j=1}^n \bar \theta_{ij}^k = \dsum_{i,j=1}^n \xi_k s_{ik} x_{ij} = \xi_k \dsum_{i=1}^n s_{ik} \dsum_{j=1}^n x_{ij} \stackrel{x \in \X_R}{=} \xi_{k} \dsum_{i=1}^n s_{ik} \stackrel{\eqref{domp2:4}}{=} \xi_k \stackrel{\eqref{domp2:1}}{\geq} \xi_{k+1} = \dsum_{i, j=1}^n \bar \theta_{ij}^{k+1}$.
\end{itemize}
Then, $\pi (\theta, \xi, s, x, (d, \bar a)) \in F_{3I}$, so $\pi(F_{2I}) \subset F_{3I}$, i.e., any solution of the convex relaxation of \eqref{domp2} induces a solution of the convex relaxation of \eqref{domp1}. Furthermore, the objective values for $(\theta, \xi, s, x, (d, \bar a))$ in \eqref{domp2} and $\pi (\theta, \xi, s, x, (d, \bar a))$ in \eqref{domp1} coincides:
\begin{align*}
\dsum_{i,j,k=1}^n \lambda_k \bar \theta_{ij}^k + \dsum_{j=1}^n f_j\bar x_{jj} &= \dsum_{i,j,k=1}^n \lambda_k \xi_k s_{ik} x_{ij} + \dsum_{j=1}^n f_j x_{jj} \\
&= \dsum_{k=1}^n \lambda \xi_k \dsum_{i=1}^n s_{ik} \dsum_{j=1}^n x_{ij} \\
&\stackrel{x \in \X_R}{=}  \dsum_{k=1}\lambda_k \xi_k \dsum_{i=1}^n s_{ik} \\
& \stackrel{\eqref{domp2:4}}{=}  \dsum_{k=1}^n \lambda_k \xi_k
\end{align*}
Thus, $z_{2I}^R \geq z_{3I}^R$.
\item Let $(u, v, D, x, (d,\bar a)) \in \R^n \times \R^n \times \R^n_+ \times \X_R \times \D$ be the optimal solution of the continuous relaxation of \eqref{domp4}. Let $p_{ik}$ the optimal dual variables associated to constraint \eqref{domp4:0}. By optimality conditions they must verify:
\begin{align*}
\dsum_{i=1}^n p_{ik}=1, \forall k=1, \ldots, n,\\
\dsum_{k=1}^n p_{ik}=1, \forall i=1, \ldots, n.\\
\end{align*}
Let us construct the following vector in $\R^n_+ \times \R^{n\times n}_+ \times [0,1]^{n\times n} \times \X_R \times \D$:
$$
\left(\bar \xi, \bar \theta, \bar s, x, (d,\bar a)\right):= \left(\left(\dsum_{i=1}^n p_{ik} D_i\right)_{k=1}^n, \left(d_{ij}x_{ij}\right)_{i,j=1}^n, \left(p_{ik}\right)_{i,k=1}^n, x, (d,\bar x)\right).
$$
By construction, $\bar s$ clearly verifies \eqref{domp2:4} and \eqref{domp2:5}. Also, note from the construction of the BEP formulation that for given $D_1, \ldots, D_n$, the problem
\begin{align*}
\min &\dsum_{i=1}^n u_i + \dsum_{k=1}^n v_k\\
\mbox{s.t} & u_i+ v_k \geq D_i,\\
& u_i, v_k \in \R, \forall i, k=1, \ldots, n.
\end{align*}
is equivalent to
\begin{align*}
\max &\dsum_{i,k=1}^n \lambda_k D_i p_{ik}\\
\mbox{s.t} & \dsum_{i=1}^n p_{ik}=1, \forall k=1, \ldots, n,\\
& \dsum_{k=1}^n p_{ik}=1, \forall i=1, \ldots, n,\\
& p_{ik} \in \{0,1\}, \forall i,k =1, \ldots, n,
\end{align*}
which is an assignment problem related to the \textit{best} sorting on the variables based on their costs given by $D_1, \ldots, D_n$. Because the monotonicity and nonnegativity of the $\lambda$-weights, this is equivalent to compute the ordered median sum $\dsum_{i=1}^n \lambda_i D_{(i)} =  \dsum_{i,k=1}^n \lambda_k D_i p_{ik}$ (where $p$ are the corresponding solution to the problem above indicating if $p_{ik}=1$ that element $i$ is sorted in the $k$th position). Hence, $\bar \xi_k = \dsum_{i=1}^n p_{ik} D_i \geq \dsum_{i=1}^n p_{ik+1} D_i = \bar \xi_{k+1}$ (constraint \eqref{domp2:1}). The proof of the verification of remainder constraints are straightforward. Also, the reader can easily check that the objective values of both solutions coincide. Thus, $z^R_{BEP} \geq z^R_{3I}$
\item Let $(u, v, D, x, (d,\bar a)) \in \R^n \times \R^n \times \R^n_+ \times \X_R \times \D$ be the optimal solution of the continuous relaxation of \eqref{domp4}. Let us construct a feasible solution for the continuous relaxation of \eqref{domp3}. Let $(\bar t, \bar z) \in \R^n \times \R^{n\times n}_+$ the solution to the problem
\begin{align*}
\min & \dsum_{k=1}^{n-1} \Delta_k (kt_k + \dsum_{i=1}^n z_{ik})\\
s.t. & z_{ik} \geq D_i - t_k, \forall i, k=1, \ldots, n,\\
& z_{ik} \geq 0, \forall i, k=1, \ldots, n,\\
& t_k \in \R, \forall k=1, \ldots, n.
\end{align*}
By the construction in Subsection \ref{formulation:OT}, the vector $(\bar t, \bar z, D, x, (d,\bar a))$ is a feasible solution to the continuous relaxation of \eqref{domp3} with same objective value than $(u, v, D, x, (d,\bar a))$ in the continuous relaxation of \eqref{domp4}, being then $z_{OT}^R \leq z_{BEP}^R$. Observe that the oposite direction can be derived with a similar reasoning. Thus, $z_{OT}^R = z^R_{BEP}$.
\end{itemize}
\end{proof}

In the above result, the relationship between $z^R_{2I}$ and $z^R_{BEP}$ (or $z^R_{OT}$) is not stated. One may think that the continuous relaxation of \eqref{domp4} is tightener than \eqref{domp2}, because the first exploits the monotonicity of the $\lambda$-weights. However, that is not always true as illustrated in the following example.

\begin{ex}
Let us consider five points in $\R^2$, $\mathcal{A}=\{(2,92),$ $(33,70),$ $(65,50),$ $(73, 69),$ $(40, 63)\}$ and neighborhoods defined as Euclidean disks with radii $\{2, 1, 0.05, 5, 1\}$. If the distance measure is the Euclidean norm and we choose as ordered median function the one with $\lambda=(1,1,1,1,1)$ and $p=2$, we get that:
$$
z^R_{3I} = 2.8348 < z^R_{OT}=z^R_{BEP} = 24.4140 <  z^R_{2I} = 34.2145 < z_{OPT}^* = 68.4751
$$
where $z_{OPT}$ is the optimal value of the OMPN. 
\end{ex}

\subsection{Computational Comparison of Relaxations}

We have run a series of experiments to study the computational performance of the formulations \eqref{domp1}, \eqref{domp2}, \eqref{domp3} and \eqref{domp4} and also to know the computational limits of the OMPN. We have randomly generated instances of $n$ demand points in $[0,100]^2$ and $[0,100]^3$ with $n$ ranging in $\{5, 6, 7, 8, 9, 10, 20, 30\}$.  Five random instances were generated for each number of points. We solved OMPN problems $p$ (number of new facilities to be located) ranging in $\{2,3,5\}$ (provided that $p<n$). Euclidean distances were considered to measure the distances between points. We considered neighborhoods defined as discs (for the planar instances) or 3-dimensional Euclidean balls (for the $3$-dimensional instances) with randomly generated radii. The sizes of the neighborhoods were generated under four scenarios:
\begin{description}
\item[Scenario 1.] Radii generated in $[0,5]$.
\item[Scenario 2.] Radii generated in $[5,10]$. 
\item[Scenario 3.] Radii generated in $[10,15]$. 
\item[Scenario 4.] Radii generated in $[15,20]$.
\end{description}
In Figure \ref{fig:radius} we show, for one of our $8$-points instances, the neighborhoods for each the four scenarios. Note that Scenario 1 slightly differs from the DOMP while Scenario 4 is closer to the continuous problem. As will be observed from our experiments, the computational difficulty of solving problems with larger radii is higher than the one of those with small radii.

\begin{figure}[h]
\input{radius1.tex}
\caption{Shapes of the neighborhoods for the different scenarios.\label{fig:radius}}
\end{figure}
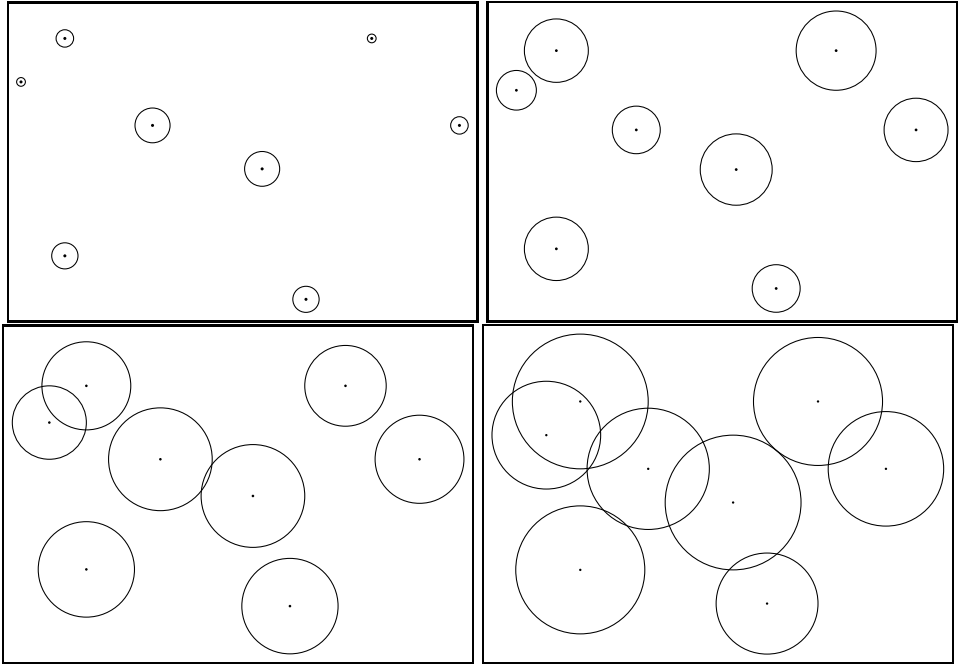
The set-up cost of each facility was assumed to be the radius of its neighborhood. It can be interpreted as the cost of covering the neighborhood (larger as $r$ increases). 

The four formulations were coded in C, and solved using Gurobi 7.01 in a Mac OSX El Capitan with an Intel Core i7 processor at 3.3 GHz and 16GB of RAM. A time limit of 1 hour was set in all the experiments.

Also, four different convex ordered median problems were solved for each of the instances:
\begin{description}
\item[$p$-Median (M):] $\lambda=(1, \ldots, 1)$.
\item[$p$-Center (C):] $\lambda=(1,0, \ldots,0)$.
\item[$p$-KCenter (K):] $\lambda=(\overbrace{1,\ldots,1}^{\lfloor\frac{n}{2}\rfloor}, 0,\ldots, 0)$.
\item[$p$-Cent-Dian$_{0.5}$ (D):] $\lambda=(1, 0.5, \ldots, 0.5)$.
\end{description}

For the small instances ($n \leq 10$) we report here only the results obtained for $n=10$ (the interested reader may check the complete result of the experiments for $n=5, \ldots, 10$ in \href{http://bit.ly/resultsDOMPN}{bit.ly/resultsDOMPN})

First, we run the continuous relaxation of the four formulations to analyze their integrality gaps. The average results are shown in Table \ref{table:IG}. We report the integrality gaps ($IG=\dfrac{z^*}{z^R}$) for each of the formulations (\eqref{domp1}, \eqref{domp2} and \eqref{domp4} (obviating \eqref{domp3} whose integrality gap coincides with the one of \eqref{domp4} by Theorem \ref{relaxations}). The table summarizes the average integrality gaps for the each of the scenarios. As remarked above, there is no order relation between $z^R_{2I}$ and $z^R_{BEP}$. We have bolfaced those values in which the average integrality gaps of \eqref{domp2} is smaller than the one for \eqref{domp4}. Note that it only happens for a few combinations of $n$, $p$, problem types and scenarios. In particular, it mostly occurs for small values of $p$ and only for Scenario 1. In Table \ref{gana2I} we show the percentage of instances (of all of those with fixed $n$ and $p$) for which $z_{2I}^R \geq z_{BEP}^R$. In the total instances, this percentage is $9.97\%$, while among the instances of Scenario 1 is $26.15\%$.

Tables \ref{table:2dtodas} and \ref{table:3dtodas} show the results for the planar and 3-dimensional problems, respectively. For each of the formulations, we provide the average CPU times (\texttt{time}), the number of nodes explored in the branch-and-bound tree (\texttt{\#nodes}) and the deviation with respect to the solution obtained at the end of the root node of the search tree (\texttt{\%gapR}).

 \input{tablesnew.tex}

  As can be observed from the results, formulations \eqref{domp3} and \eqref{domp4} are much less time consuming than \eqref{domp1} and \eqref{domp2} in all the cases. Also, the solutions obtained after exploring the root node of \eqref{domp3} and \eqref{domp4} are tighten than the rest. Consequently, the number of explored nodes to find the optimal solution or to certify optimality is higher in the two first formulations. Observe that the results are as expected since the first two formulations do not exploit the convexity of monotone ordered median problems. Observe that the \textit{sorting constraints} in the first two formulations involve binary variables while in the two last formulations no need of new binary variables are needed for this task.
  
Since \eqref{domp3} and \eqref{domp4} seems to have a similar computational behavior for the small-size instances, we have performed a series of experiments for medium-size instances to compare these two formulations. The results are shown in Table \ref{table:OT-BEP}, where now $n$ ranges in $\{20,30\}$ and $p$ in $\{2,5,10\}$. As can be observed, the performance (in terms of CPU time) of both formulation is similar, but \eqref{domp4} seems to need, in average, less CPU time to solve the problems in most of the problems, and the standard deviations (StDev) of the consuming times for \eqref{domp4} are smaller than those for \eqref{domp3}. Furthermore, we were able to solve all the instances before the 1 hour time limit, but $2.56\%$ of them by using \eqref{domp4}, while \eqref{domp3} was not able to solve $11.34\%$ of the the problems. Moreover, in all the instances, \eqref{domp4} obtained better upper bounds for the optimal value of the problems in all the instances (the deviation of the best upper bounds obtained with the OT formulation with respect to the best solution obtained with the BEP formulation is shown in column \%DevBest).

  {\small  \begin{table}[h]
  \centering
  {\tiny
    \begin{tabular}{|c|c|c|cc|cc|cc|c|}\hline
 Sc.&  $n$ & $p$ & Time$_{BEP}$ & StDev$_{BEP}$ & Time$_{OT}$ & StDev$_{OT}$ & \%NonSolved$_{BEP}$ & \%NonSolved$_{OT}$ & \%DevBest \\\hline
\multirow{6}{*}{3}     &    \multirow{3}{*}{20}     &          2    & 9.73 & 2.71 & 11.36 & 3.58 & 0\% & 0\% & 0\% \\
 &&5    & 253.35 & 18.65 & 449.32 & 28.15 & 0\% & 2.56\% & 0.01\% \\
 &&10   & 46.97 & 13.85 & 77.52 & 17.47 & 0\% & 0\%& 0\% \\\cline{2-10}
        &    \multirow{3}{*}{30}     &2    & 59.64 & 6.48 & 148.23 & 14.70 & 0\% & 0\% & 0\% \\
 &&5    & 2931.44 & 36.11 & 3099.25 & 32.91 & 75\% & 77.5\% & 1.63\% \\
 &&10   & 2861.03 & 37.34 & 3070.86 & 34.40 & 77.5\% & 80\% & 3.75\% \\\hline
   \multirow{6}{*}{4}     &    \multirow{3}{*}{20}     &2    & 26.45 & 4.40 & 30.03 & 4.95 & 0\% & 0\% & 0\% \\
 &&5    & 1865.88 & 41.15 & 1874.01 & 41.33 & 40\% & 40\% & 0.30\% \\
 &&10   & 9.51 & 2.44 & 22.13 & 6.17 & 0\% &0\% & 0\% \\\cline{2-10}
   \multirow{6}{*}{3}     &    \multirow{3}{*}{30}     &2    & 735.58 & 36.69 & 849.51 & 37.10 & 15\% & 15\% & 0.17\% \\
 &&5    & 2742.49 & 39.04 & 2836.95 & 36.91 & 75\% & 75\% & 1.15\% \\
 &&10   & 2745.59 & 39.12 & 2789.27 & 38.12 & 75\% & 75\% & 3.28\%\\\hline
           \end{tabular}}
          \caption{Comparison of \eqref{domp3} and \eqref{domp4} for instances with $n=20, 30$.\label{table:OT-BEP}}
  \end{table}  }
  
\section{Math-heuristics for the OMPN}
\label{heuristics}

In this section we describe two mathematical programming location-allocation based heuristic approaches for solving the OMPN for larger instances. Some heuristics have been proposed for solving ordered $p$-median problems (see \cite{DNHM_AOR05}). However, most of them are based on the use of ``fast'' procedures to compute the overall cost of opening/closing certain sets of facilities. Note that when the travel cost matrix is provided and a set of open facilities is obtained, one can easily evaluate, for each customer, its cheapest facility (or its second cheapest facility), and once all of them are computed, evaluate an ordered median function can be also efficiently performed. In the OMPN, even if the open facilities are known, the allocation costs depend on the final location of the facilities (which depends also on the customers allocated to each facility). Hence, the developed heuristics for the DOMP are no longer valid for the OMPN problem. We propose two alternative local search math-heuristic which allows us to solve larger instances of the problem at  smaller computational costs than the exact approaches, at the price of not warrantying the optimality of the solution.

The two heuristics procedures that we propose are part of the well-known familiy of location-allocation procedures. These type of schemes based on performing changes over a set of $p$ facilities candidates for being opened, trying to improve the incumbent objective value. For the sake of that we need to compute the overall cost of opening a given set of facilities $J \subseteq \{1, \ldots, n\}$ with $|J|=p$. Observe that if a set of open facilities is known, we have to compute the minimum (ordered weighted) cost of allocating the customers to those facilities. As mentioned above the computation of such a cost will involve the computation of the allocation customer-open facility and also the position of the open facilities inside their neighborhoods. Although different formulations can be used for such a task, we will use the one based on  formulation \eqref{domp4}.

\begin{propo}
Let $J \subset N=\{1, \ldots, n\}$ with $|J|=p$. Then, the cost of using $J$ as open facilities of the OMPN problem can be computed by solving the following mixed integer non linear programming problem:

\begin{align}
c(J) := \min& \dsum_{i \in N\backslash J} u_i + \dsum_{k \in N \backslash J} v_k + \dsum_{j \in J} f_j,\label{cost:J}\tag{${\rm ALLOC}(J)$}\\
\mbox{s.t. } & u_i + v_k \geq \lambda_k z_i, \forall i, k \in N \backslash J,\nonumber\\
& z_i \geq d_{ij} - \widehat{D}_{ij} (1-x_{ij}), \forall i \in N \backslash J, j \in J,\nonumber\\
& d_{ij} \geq \|a_i - \bar a_j\|, \forall i \in N \backslash J, j \in J,\nonumber\\
& r_j \geq \|a_j - \bar a_j\|, \forall j \in J,\nonumber\\
& \dsum_{j \in J} x_{ij} = 1, \forall i \in N\backslash J,\nonumber\\
& x_{ij} \in \{0,1\}, \forall i \in N \backslash J, j \in J,\nonumber\\
&z_i \geq 0,\forall  i \in N \backslash J.\nonumber
\end{align}
\end{propo}
\begin{proof}
The proof easily follows noting that \eqref{cost:J} is nothing but the simplification of \eqref{domp4} when the values of $x_{jj}$ are known and fixed to $x_{jj}=1$ if $j \in J$ and $x_{jj}=0$, otherwise.
\end{proof}

For each $J$, \eqref{cost:J} can be reformulated as a mixed integer second order cone constraint problem with $(n-p)p$ binary variables (instead of the $n^2$ in \eqref{domp4}. Furthermore, a variable fixing strategy can be applied to \eqref{cost:J} in order to reduce the number of binary variables of the problem.

\begin{prop}
Let $J \subseteq N$ with $|J|=p$, $i\in N\backslash J$, $j\in J$ and $x^*\in\X$ optimal allocation solutions of \eqref{cost:J}.
\begin{enumerate}
\item\label{1}If $\exists k \in J\backslash\{j\}$ such that $\widehat{D}_{ik} < \widehat{d}_{ij}$, then $x^*_{ij}=0$.
\item If $\min_{k \neq j} \widehat{d}_{ik} > \widehat{D}_{ij}$, then $x^*_{ij}=1$.
\item If $\{j^\prime \in J: \widehat{D}_{ik} \geq \widehat{d}_{ij^\prime}, \forall k \neq j^\prime\}= \{j\}$, then $x^*_{ij}=1$.
\end{enumerate}
\end{prop}
\begin{proof}$ $
\begin{enumerate}
\item Let us assume that $x^*_{ij}=1$. Then, $d_{ij}^* = \|a_i - \bar a^*_j\| \geq \widehat{d}_{ij}$. By hypothesis there exists $k \in J$ ($k \neq j$) with $\widehat{D}_{ik} < \widehat{d}_{ij}$. Hence, $d_{ij}^* > \widehat{D}_{ik} \geq d_{ik} = \|a_i - \bar a_k^*\|$, so $d_{ij}^* \neq \min_{j^\prime \in J} \|a_i - \bar a_k\|$ contradicting the optimality of the solution.
\item If $x^*_{ij}=0$, then, there exists $k \in J$, $k\neq j$ such that $d_{ij}^* \geq d_{ik}^* > \widehat{D}_{ij} \geq d_{ij}^*$. Thus, $x^*_{ij}=1$.
\item If applying \ref{1}, all the facilities except $j$ must verify $x_{ij^\prime}^*=0$, then the unique choice for allocating $i$ is $j$. 
\end{enumerate}
\end{proof}

As we will show in our computational experiments, the above strategies for fixing variables allows us to fix an average of $80\%$ of the binary variables in the test problems.

Using the above-described formulation, we implemented two different heuristic algorithms. Both algorithms will move through different feasible solutions in order to improve an initial feasible solution. This initial solution is constructed by either solving the standard DOMP problem with a set of weights based on the distances between centers of the neighborhoods (a convex combination of $\widehat{D}_{ij}$ and $\widehat{d}_{ij}$), or solving the OMPN problem for simpler neighborhoods (as polyhedral neighborhoods) and polyhedral distances (which may require less computational effort than general $\ell_\tau$-norm based metrics or neighborhoods). Hence, we consider that an initial solution $x^0 \in \X$ is known.

\subsection{Math-heuristic Algorithm 1}

Given a feasible solution $\bar x \in \X$, the first algorithm searches, for each facility $j_0$ in $J$, the best replacement by a facility in $N \backslash J$. Two different options are possible here. First, to construct the new set of open facilities $J^\prime= J \cup \{i\} \backslash \{j_0\}$ for each $i \in N\backslash J$, solve \eqref{cost:J} for such a $J^\prime$ and keep the best possible change for $j_0$. The second option is to solve a single mixed integer non linear programming problem  which decides (through the binary variable $\xi_i$, whether the non-opened facility $i$ is interchanged by $j_0$ to obtain the best improvement:

\begin{align}
\min& \dsum_{i \in N} u_i + \dsum_{k \in N} v_k + \dsum_{j\in J \backslash\{j_0\}} f_j + \dsum_{i \in N \backslash J} f_i \xi_i,\label{r:j}\tag{${\rm BestRepl}(j)$}\\
\mbox{s.t. } & u_i + v_k \geq \lambda_k z_i, \forall i, k \in N \backslash J,\nonumber\\
& z_i \geq d_{ij} - \widehat{D}_{ij} (1-x_{ij}), \forall i \in \{j_0\} \cup N \backslash J, j \in J\backslash\{j_0\},\label{b:1}\\
& z_i \geq d_{ij} - \widehat{D}_{ij} (2-x_{ij}-\xi_j), \forall i \in \{j_0\} \cup N \backslash J, j \in J\backslash\{j_0\},\label{b:2}\\
& d_{ij} \geq \|a_i - \bar a_j\|, \forall i \in \{j_0\} \cup N \backslash J, j \neq j_0\in J,\nonumber\\
& r_j \geq \|a_j - \bar a_j\|, \forall j\neq \{j_0\}  \in N,\nonumber\\
& \dsum_{j \in N \backslash\{j_0\}} x_{ij} = 1, \forall i \in \{j_0\} \cup N \backslash J,\nonumber\\
& x_{ij} \leq \xi_j, \forall j \in N\backslash J,\label{b:3}\\
& \dsum_{i \in N\backslash J} \xi_i = 1, \label{b:4}\\
& x_{ij}, \xi_i \in \{0,1\}, \forall i \in \{j_0\} \cup N \backslash J, j \in N \backslash\{j_0\},\nonumber\\
&z_i \geq 0,\forall  i \in N \backslash J.\nonumber
\end{align}

Note that constraints \eqref{b:1} are the linearization of the bilinear terms as in the previus formulations, but obviating the facility that wants to be replaced ($j_0$). For the candidates to replace $j_0$, constraints \eqref{b:2} assures that in case $j$ is chosen for the replacement, and a customer $i$ is allocated to $j$, then the travel cost for $i$ is $d_{ij}$, otherwise, the constraint is redundant. With respect to the variables $\xi$ that model the selection of the facility to be swapped with $j_0$, \eqref{b:3} ensures that unchosen facilities cannot serve any customer. \eqref{b:4} assures that a single choice for $j_0$ is possible. Although \eqref{r:j} is similar to \eqref{domp4}, the number of binary variables in the problemiss $(n-p)n$ instead of $n^2$.

In our experiments, we have checked that solving \eqref{r:j} required more CPU time than solving the $n-p$ problems in the form \eqref{cost:J}, although for problems in which $n-p \ll n$ (i.e., when $p$ is large), the \textit{compact} formulation may consume less CPU time than loading and solving $n-p$ problems in the shape of \eqref{cost:J}.

In what follows we describe our math-heuristic procedure, whose pseudocode is shown in Algorithm \ref{alg:heuristic}. Given an initial set of $p$ open facilities, it iterates by interchanging open facilities with other potential facilities trying to improve the best upper bound. At each iteration an open facility is selected to be replaced and the best replacement is chosen. After checking all the open facilities, if an improvement is found when compared to the best upper bound, the latest and the set of open facilities are updated. The procedure repeats the same scheme until a termination criterion is fulfilled. In our case, two stopping criteria are considered: maximum number of iterations and maximum number of iterations without improvement in the solution. In order to reduce the computation times required for solving \eqref{cost:J} or \eqref{r:j}, we consider a randomized version of the algorithm in which instead of finding best replacements for all the open facilities, a random one is selected at that phase of the approach. 

\begin{algorithm}[H]
\SetKwInOut{Input}{Initialization}\SetKwInOut{Output}{output}

 \Input{Let $\widehat{J} \subset N$ with $|\widehat{J}|=p$ an initial set of open facilities and $UB=c(\widehat{J})$.}

 \While{$it<it_{max}$}{
 \For{$j_0\in \widehat{J}$}{
 Find the best replacement for $j$ (by solving \eqref{cost:J} for $J=\widehat{J} \cup \{i\} \backslash \{j\}$ or \eqref{r:j}: $c_{j_0} = c(J)$.
}
 
 \If{$c_j <UB$}{
 Update $UB=c_j$\\
 $\widehat{J}=J$\\
 BREAK}

Increase $it$.
}
 \caption{Math-Heuristic 1 for solving OMPN.\label{alg:heuristic}}
\end{algorithm}

A crucial point of local search heuristics is the quality of an initial feasible solution (in the $x$-variables). We compute the solution of the DOMP problem  but using at costs between facilities $i$ and $j$ a convex combination of the lower and upper bounds $\widehat{d}_{ij}$ and $\widehat{D}_{ij}$ which provide good results in practice.

\subsection{Math-Heuristic Algorithm 2}

The second heuristic is based on alternating the location and allocation decisions. Initially,  a DOMP is solved by fixing $\bar a = a$, and precomputing the distances between the facilities. Once a solution is obtained, the optimal open facilities are kept and given as input to \eqref{cost:J}. Then, the variables $\bar a$ are updated with the obtained solution and the process is repeated until stabilization. In order to escape from local optima, the scheme is applied again but forbidding the use of one of the facilities opened in the first stage. The process iterates until no improvements are found.

The pseudocode for this approach is shown  in Algorithm \ref{alg:h2}.

\begin{algorithm}[H]
\SetKwInOut{Input}{Initialization}\SetKwInOut{Output}{output}

\Input{$\bar a = a$}

\While{$|f_1 - f_2| > \varepsilon$}{
$\bullet$ Solve DOMP for $d_{ij} = \| a_i - \bar a_j\|$. Update $J = \{j\in N: x^*_{jj} =1\}$ and its objective value $f_1$.\par
$\bullet$ Solve \eqref{cost:J} and update $\bar a$.
}

\For{$j_0 \in J$}{
Initialize $\bar a = a$.

\While{$|f_1 - f_2| > \varepsilon$}{
$\bullet$ Solve DOMP for $d_{ij} = \| a_i - \bar a_j\|$ forbiding opeing $j_0$. Update $J = \{j\in N: x^*_{jj} =1\}$ and its objective value $f_1$.\par
$\bullet$ Solve \eqref{cost:J} and update $\bar a$ and its objective value $f_2$.
}
}

 \caption{Math-Heuristic 2 for solving OMPN.\label{alg:h2}}
\end{algorithm}

\section{Experiments}

In order to test the performance of the math-heuristic approaches, we have run some experiments over the real dataset instance of 2-dimensional coordinates (normalized longitude and latitude) of geographical centers of 49 states of the Unites States (we exclude Alaska, Hawaii and those outside Northamerica). We considered as neighborhoods Euclidean disks with radii based on the areas of each state. For each state (indexed by $j$), the area (in km$^2$), $A_j$, was obtained and we construct the radius $r^0_j = \sqrt{\frac{A_j}{\pi}}$.  The coordinates and the discs built applying this strategy are drawn in Figure \ref{us}. Then, three different scenarios were considered:
\begin{description}
\item[S1]: $r_j = r^0_j$, for $j=1, \ldots 49$.
\item[S2]: $r_j = 2\times r^0_j$, for $j=1, \ldots 49$.
\item[S3]: $r_j = 3\times r^0_j$, for $j=1, \ldots 49$.
\end{description}

The interested reader may download the datasets at \url{http://bit.ly/datasetUS}.

\input{usr1.tex}

We implemented in Gurobi under the C API the two math-heuristic approaches and we compare the running times and the best values obtained with these procedures and those obtained with the exact \eqref{domp4} formulation (with a time limit of 1 hour). We solved the $p$-Median (M), $p$-Center (C), $p$-$25$-center (K) and $p$-Centdian (D) with $p \in \{2, 5, 10\}$. The results are reported in Table \ref{res:heuristics}. In such a table, the first column indicates the scenario ($1$, $2$ or $3$), the second column (Pr.) shows the problem type and the third column indicates the number of facilities to be open, $p$. The values of the solutions obtained by using the different aproaches as well as their CPU running times (in seconds) are reported: 
\begin{itemize}
\item Initial solution obtained by solving the nominal DOMP problem and solving \eqref{cost:J} for the obtained open facilities: \texttt{H0} and \texttt{t0}.
\item Best solution obtained by the math-heuristic approach 1: \texttt{H1} and \texttt{t1}.
\item Best solution obtained by the math-heuristic approach 2: \texttt{H2} and \texttt{t2}.
\item Best solution obtained by exact formulation \eqref{domp4} within the time limit: \texttt{BEP} and \texttt{tBEP}.
\end{itemize}
We also report in the 12th column (\texttt{\%VarFixed}) the average number of binary variables fixed in the first heuristic, and the pertentage deviations of the obtained solutions with respect to the best solution found with the exact formulation within the time limit: \texttt{G1}, \texttt{G2} and \texttt{G0} for the first heuristic, the second heuristic and the initial solution, respectively.

\begin{landscape}
\begin{table}[h]
{\small \centering
    \begin{tabular}{|c|c|c|cccc|cccc|c|ccc|}\hline
   \texttt{Sc.} &  \texttt{$p$} &   \texttt{Pr.}  & \texttt{H0} &  \texttt{H1} &  \texttt{H2} &  \texttt{BEP} &  \texttt{t0} &  \texttt{t1} &  \texttt{t2} &  \texttt{tBEP} &  \texttt{\%VarFixed} &  \texttt{G1} &  \texttt{G2} &  \texttt{G0} \\\hline
       \multirow{12}{*}{SC1}     & \multirow{4}{*}{2}      &  M      & 395.3482 & 394.8909 & 395.3482 & 394.891 & 2.5   & 31.52 & 13.28 & 32.44 & 89.15\% & 0.00\% & 0.12\% & 0.12\% \\
          &     &  C      & 19.105 & 18.0278 & 18.0278 & 18.0278 & 0.08  & 0.28  & 4.2   & 8.18  & 90.41\% & 0.00\% & 0.00\% & 5.64\% \\
         &      &  K      & 272.8855 & 270.8245 & 270.7348 & 270.7348 & 0.84  & 12.88 & 33.24 & 22.92 & 89.84\% & 0.03\% & 0.00\% & 0.79\% \\
         &      &  D      & 207.2299 & 207.2299 & 207.2299 & 207.2298 & 2.99  & 18.14 & 12.41 & 22.12 & 89.88\% & 0.00\% & 0.00\% & 0.00\% \\\cline{2-15}
         & \multirow{4}{*}{5}      &  M      & 222.3974 & 221.607 & 221.5985 & 222.6594 & 2.68  & 56.03 & 160.26 &  $>3600$  & 81.54\% & -0.47\% & -0.48\% & -0.12\% \\
         &      &  C      & 18.619 & 18.2246 & 16.2491 & 16.2491 & 0.17  & 0.35  & 9.57  & 37.59 & 84.90\% & 10.84\% & 0.00\% & 12.73\% \\
         &      &  K      & 167.6711 & 163.564 & 160.1906 & 160.1138 & 3.98  & 24.07 & 239.04 &  $>3600$  & 83.05\% & 2.11\% & 0.05\% & 4.51\% \\
         &      &  D      & 120.9466 & 120.0572 & 119.8601 & 119.7391 & 2.66  & 67.81 & 116.23 &  $>3600$  & 83.02\% & 0.27\% & 0.10\% & 1.00\% \\\cline{2-15}
         & \multirow{4}{*}{10}     &  M      & 146.2992 & 144.6164 & 144.2134 & 141.8635 & 8.4   & 145.96 & 245.51 &  $>3600$  & 85.33\% & 1.90\% & 1.63\% & 3.03\% \\
         &     &  C      & 27.0357 & 22.9692 & 19.9944 & 19.9944 & 0.19  & 3.06  & 15.9  & 55.57 & 84.83\% & 12.95\% & 0.00\% & 26.04\% \\
         &     &  K      & 118.7461 & 118.0626 & 117.7095 & 125.0503 & 6.34  & 70.23 & 1760.77 &  $>3600$  & 85.80\% & -5.92\% & -6.24\% & -5.31\% \\
         &     &  D      & 86.6926 & 85.2866 & 85.7954 & 85.5946 & 10.84 & 94.08 & 357.67 &  $>3600$  & 86.32\% & -0.36\% & 0.23\% & 1.27\% \\\hline
    \multirow{12}{*}{SC2}      & \multirow{4}{*}{2}      &  M      & 399.6586 & 395.1866 & 398.2659 & 395.1789 & 2.58  & 96.14 & 27.78 & 462.06 & 77.53\% & 0.00\% & 0.78\% & 1.12\% \\
         &      &  C      & 23.3894 & 21.8305 & 22.1818 & 21.7935 & 0.11  & 0.45  & 6.72  & 10.48 & 85.33\% & 0.17\% & 1.75\% & 6.82\% \\
         &      &  K      & 275.9869 & 274.8279 & 274.6485 & 274.2601 & 1.4   & 25.67 & 32.68 & 365.95 & 78.30\% & 0.21\% & 0.14\% & 0.63\% \\
         &      &  D      & 211.7095 & 209.7885 & 211.8238 & 209.7885 & 2.59  & 138.32 & 56.4  & 262.4 & 75.94\% & 0.00\% & 0.96\% & 0.91\% \\\cline{2-15}
         & \multirow{4}{*}{5}      &  M      & 228.6722 & 221.514 & 226.7708 & 223.3972 & 637.69 &  $>3600$  & 689.41 &  $>3600$  & 67.96\% & -0.85\% & 1.49\% & 2.31\% \\
         &      &  C      & 29.0877 & 23.0926 & 22.0907 & 21.5931 & 0.55  & 1.68  & 15.67 & 46.04 & 71.11\% & 6.49\% & 2.25\% & 25.77\% \\
         &      &  K      & 171.2099 & 168.7613 & 167.6214 & 166.7242 & 21.55 & 163.7 & 1025.78 & $>3600$  & 64.07\% & 1.21\% & 0.54\% & 2.62\% \\
         &      &  D      & 129.7708 & 125.1129 & 125.4487 & 127.2981 & 266.3 & 600.94 & 641.76 &  $>3600$  & 69.92\% & -1.75\% & -1.47\% & 1.91\% \\\cline{2-15}
         & \multirow{4}{*}{10}     &  M      & 153.013 & 153.013 & 155.285 & 151.5515 & 3600.1 &  $>3600$  & 3208.72 &  $>3600$  & 76.33\% & 0.96\% & 2.40\% & 0.96\% \\
         &     &  C      & 44.5115 & 35.0981 & 29.8835 & 29.842 & 0.66  & 16.66 & 46.47 & 53.42 & 74.23\% & 14.98\% & 0.14\% & 32.96\% \\
         &     &  K      & 132.6658 & 132.6658 & 131.9971 & 142.0924 &  $>3600$     &  $>3600$  &  $>3600$  &  $>3600$  & 76.33\% & -7.11\% & -7.65\% & -7.11\% \\
         &     &  D      & 100.474 & 100.474 & 97.6098 & 108.2469 &  $>3600$     &  $>3600$  & 3215.14 &  $>3600$  & 76.33\% & -7.74\% & -10.90\% & -7.74\% \\\hline
    \multirow{12}{*}{SC3}      & \multirow{4}{*}{2}      &  M      & 404.466 & 394.9373 & 398.7388 & 395.7311 & 2.94  & 108.55 & 49.53 &  $>3600$  & 62.96\% & -0.20\% & 0.75\% & 2.16\% \\
         &      &  C      & 28.0977 & 24.1051 & 24.1924 & 24.1051 & 0.11  & 0.66  & 6.03  & 10.05 & 74.69\% & 0.00\% & 0.36\% & 14.21\% \\
         &      &  K      & 280.5643 & 278.0229 & 278.9015 & 278.0228 & 1.42  & 32.54 & 41.81 & 1455.24 & 52.65\% & 0.00\% & 0.32\% & 0.91\% \\
         &      &  D      & 216.5167 & 210.8067 & 214.7534 & 210.6248 & 2.97  & 74.29 & 83.65 & 1488.5 & 72.21\% & 0.09\% & 1.92\% & 2.72\% \\\cline{2-15}
         & \multirow{4}{*}{5}      &  M      & 232.5595 & 223.4184 & 235.2989 & 240.8416 & 915.52 &  $>3600$  & 702.09 &  $>3600$  & 60.38\% & -7.80\% & -2.36\% & -3.56\% \\
         &      &  C      & 34.7835 & 28.4267 & 27.4501 & 26.5732 & 0.23  & 3.57  & 13.95 & 36.93 & 64.23\% & 6.52\% & 3.19\% & 23.60\% \\
         &      &  K      & 183.8955 & 171.2615 & 178.0249 & 173.4423 & 128.12 & 1912.6 & 1310.85 &  $>3600$  & 54.76\% & -1.27\% & 2.57\% & 5.68\% \\
         &      &  D      & 134.1824 & 133.0089 & 132.521 & 138.3565 & 1069.6 &  $>3600$  & 1206.69 &  $>3600$  & 57.96\% & -4.02\% & -4.40\% & -3.11\% \\\cline{2-15}
         & \multirow{4}{*}{10}     &  M      & 167.5751 & 167.5751 & 171.1057 & 177.7685 &  $>3600$     &  $>3600$  & 3208.25 &  $>3600$  & 70.41\% & -6.08\% & -3.89\% & -6.08\% \\
         &     &  C      & 50.7627 & 47.0082 & 39.3866 & 38.5431 & 1.1   & 4.02  & 25.13 & 32.2  & 70.00\% & 18.01\% & 2.14\% & 24.07\% \\
         &     &  K      & 144.9262 & 144.9262 & 146.6831 & 161.303 &  $>3600$    &  $>3600$  & 2336.69 &  $>3600$  & 70.41\% & -11.30\% & -9.97\% & -11.30\% \\
        &    &  D      & 110.2827 & 110.2827 & 110.5119 & 116.5976 &  $>3600$   &  $>3600$  & 3214.45 &  $>3600$  & 70.41\% & -5.73\% & -5.51\% & -5.73\% \\\hline
         \end{tabular}
  \label{res:heuristics}
  \caption{Results of Math-Heuristic Approaches and \eqref{domp4} in the US dataset.  \label{res:heuristics}}}
  \end{table}
  \end{landscape}
  
  One can observe from the results that, the CPU times needed to run the math-heuristic approaches are much smaller than those needed to solve the OMPN problem with the MINLP formulation. In those cases in which all the approaches were able to solve the problem before the time limit of one hour, the highest deviation with respect to the optimal solutions was $15\%$ for the first heuristic and $3.2\%$ for the second one. In those cases in which the exact approach was not able to certify optimality in one hour, the math-heuristic approaches found a better solution for the problem. In the first heuristic, we apply the fixing variables strategy each time \eqref{cost:J} is solved. The average percentage of binary variables that are fixed with this strategy, is at least $84\%$ for scenario SC1, $75\%$ for SC2 and $52\%$ for SC3. Observe also that the initial solution based on fixing the open facilities to the solution of the DOMP problem and then compute the location on the neighborhoods and the allocation of the customers according to these positions, is in some case far of being a close-to-optimal choice, with percentage deviations of $33\%$ in some cases.
  
  Note that the two math-heuristic approaches are still very time consuming.  One may not forget that both proposed approaches are based on solving mixed integer non linear programming problems which are known to be NP-hard. The advantage of the two approaches is that they provided good quality solutions at the first iterations, which are competitive with the exact solutions (in terms of gap).

  In Figure \ref{fig:solC5} we show the best solutions obtained for the test problem for $p=5$ under the center objective function for scenario SC1. The initial solution for this problem is drawn in Figure \ref{fig:solC05} (the solutions for $p=2, 5, 10$ can be found in \href{http://bit.ly/resultsDOMPN}{bit.ly/resultsDOMPN}). The reader can observe that \textit{small} modification of the coordinates of the potential facilities (through neighborhoods) may produce different location-allocation solutions.

 \input{usr1-solC.tex}

 \input{usr1-solC0.tex}

\section{Conclusions}
\label{sec:conclusions}

A unified version of the classical $p$-median problem is introduced in this paper which includes as particular cases the discrete and the continuous $p$-median and $p$-center problems. The problem considers that each facility can be located not only in the exact given position but in a neighborhood around it. Also, ordered median objective functions are modeled for the problem. Several mathematical programming formulations are proposed based on formulations for the discrete ordered median problem obtained from different sources. 
 Two location-allocation approaches for solving the problem are presented. Although the optimization problems needed to solve are still NP-hard, the reduced dimensionality of them allows us to provide good quality solution in more reasonable times.

Several extensions are posible within this new framework. The first is the development of decomposition approaches for solving the OMPN. Lagrangean decomposition (relaxing the ordering constraints) combined with Benders decomposition (to \textit{separate} the discrete and the continuous decisions) may produce exact solutions in better CPU times. On the other hand,  although we analyze the ordered $p$-median problem, the results in this paper can be extended to other discrete location problems. For instance, capacitated \cite{P_ORP08} or multiperiod \cite{AFHP_CORS09,NS_chapter2015} location problems can be embedded into the neighborhoods framework. Other interesting related problem which is left for further research is the consideration of location-routing problems with neighborhoods. That problem would involve not only the discrete facility location problem with neighborhoods but also the TSP with neighborhoods, then, the combination of the methods proposed in this paper with those provided in \cite{GMS_OMS13} may be applicable to the problem. Also, the case in which the neighborhood of each facility is the union of convex sets would be an interesting next step within this framework. In particular, it would model the case in which two facilities may belong to the same neighborhood. The extended MINLP formulations for such a problem will become disjunctive MINLP for which some techniques are available in the literature. Another approach that would extend the version introduced in this paper is the one in which $k_j$ facilities are allowed to be located at the $j$-th neighborhood to allocate the demand points. In such a case, a nested multifacility $p$-median problem is considered for which more sophisticated strategies should be developed to solve even small-size instances.

\section*{Acknowledgements}

The author was partially supported by project MTM2016-74983-C2-1-R (MINECO, Spain), the research group SEJ-534 (Junta de Andaluc\'ia) and the research project PP2016-PIP06 (Universidad de Granada).

%\section*{References}

\bibliographystyle{amsplain} 
%\bibliography{domp}

\providecommand{\bysame}{\leavevmode\hbox to3em{\hrulefill}\thinspace}
\providecommand{\MR}{\relax\ifhmode\unskip\space\fi MR }
\providecommand{\MRhref}[2]{%
  \href{http://www.ams.org/mathscinet-getitem?mr=#1}{#2}
}
\providecommand{\href}[2]{#2}

\end{document}

%% file: ex0-1.tex
\begin{center}
\fbox{\begin{tikzpicture}[scale=0.35]
%Coordinates of the initial nodes of the Graph
\coordinate(X1) at (0,5);
\coordinate(X2) at (1,1);
\coordinate(X3) at (1,6);
\coordinate(X4) at (3,4);
\coordinate(X5) at (5.5,3);
\coordinate(X6) at (10,4);
\coordinate(X7) at (6.5,0);
\coordinate(X8) at (8,6);
%%2-median
\draw[gray] (X2)--(X1);
\draw[gray] (X3)--(X1);
\draw[gray] (X4)--(X5);
\draw[gray] (X6)--(X5);
\draw[gray] (X7)--(X5);
\draw[gray] (X8)--(X5);
%%Points
%\fill (X1) circle (4pt);
%\ngram{4}{5}{45}{thick,fill=red}
\node[fill,star,star points=5, scale=0.5] at (X1) {};
\fill (X2) circle (2pt);
\fill (X3) circle (2pt);
\fill (X4) circle (2pt);
\node[fill,star,star points=5, scale=0.5] at (X5) {};
%\fill (X5) circle (4pt);
\fill (X6) circle (2pt);
\fill (X7) circle (2pt);
\fill (X8) circle (2pt);
\end{tikzpicture}}
\fbox{\begin{tikzpicture}[scale=0.35]
%Coordinates of the initial nodes of the Graph
\coordinate(X1) at (0,5);
\coordinate(X2) at (1,1);
\coordinate(X3) at (1,6);
\coordinate(X4) at (3,4);
\coordinate(X5) at (5.5,3);
\coordinate(X6) at (10,4);
\coordinate(X7) at (6.5,0);
\coordinate(X8) at (8,6);
%%2-center
\draw[gray] (X1)--(X4);
\draw[gray] (X2)--(X4);
\draw[gray] (X3)--(X4);
\draw[gray] (X6)--(X5);
\draw[gray] (X7)--(X5);
\draw[gray] (X8)--(X5);
%%Points
\fill (X1) circle (2pt);
\fill (X2) circle (2pt);
\fill (X3) circle (2pt);
%\fill (X4) circle (4pt);
\node[fill,star,star points=5, scale=0.5] at (X4) {};
%\fill (X5) circle (4pt);
\node[fill,star,star points=5, scale=0.5] at (X5) {};
\fill (X6) circle (2pt);
\fill (X7) circle (2pt);
\fill (X8) circle (2pt);
\end{tikzpicture}}
\fbox{\begin{tikzpicture}[scale=0.35]
%Coordinates of the initial nodes of the Graph
\coordinate(X1) at (0,5);
\coordinate(X2) at (1,1);
\coordinate(X3) at (1,6);
\coordinate(X4) at (3,4);
\coordinate(X5) at (5.5,3);
\coordinate(X6) at (10,4);
\coordinate(X7) at (6.5,0);
\coordinate(X8) at (8,6);
%%2-4center
\draw[gray] (X1)--(X4);
\draw[gray] (X2)--(X4);
\draw[gray] (X3)--(X4);
\draw[gray] (X5)--(X4);
\draw[gray] (X7)--(X6);
\draw[gray] (X8)--(X6);
%%Points
\fill (X1) circle (2pt);
\fill (X2) circle (2pt);
\fill (X3) circle (2pt);
%\fill (X4) circle (4pt);
\node[fill,star,star points=5, scale=0.5] at (X4) {};
\fill (X5) circle (2pt);
%\fill (X6) circle (4pt);
\node[fill,star,star points=5, scale=0.5] at (X6) {};
\fill (X7) circle (2pt);
\fill (X8) circle (2pt);
\end{tikzpicture}}
\end{center}

%% file: ex0-2.tex
\begin{center}
\fbox{\begin{tikzpicture}[scale=0.65]

%Coordinates of the initial nodes of the Graph

\coordinate(X1) at (0,5);
\coordinate(X2) at (1,1);
\coordinate(X3) at (1,6);
\coordinate(X4) at (3,4);
\coordinate(X5) at (5.5,3);
\coordinate(X6) at (10,4);
\coordinate(X7) at (6.5,0);
\coordinate(X8) at (8,6);

\draw (X1) circle (1.0);
\draw (X2) circle (0.6);
\draw (X3) circle (1);
\draw (X4) circle (0.6);
\draw (X5) circle (2.4);
\draw (X6) circle (2.4);
\draw (X7) circle (0.8);
\draw (X8) circle (1.6);

\fill (X1) circle (1pt);
\fill (X2) circle (1pt);
\fill (X3) circle (1pt);
\fill (X4) circle (1pt);
\fill (X5) circle (1pt);
\fill (X6) circle (1pt);
\fill (X7) circle (1pt);
\fill (X8) circle (1pt);

%\fill[blue] (Y1) circle (2pt);
%\fill[blue] (Y5) circle (2pt);

%% 2-medianN
%
%\draw[gray] (X2)--(Y1);
%\draw[gray] (X3)--(Y1);
%\draw[gray] (X4)--(Y1);
%\draw[gray] (X6)--(Y5);
%\draw[gray] (X7)--(Y5);
%\draw[gray] (X8)--(Y5);
\end{tikzpicture}}
\end{center}

%% file: ex0-3.tex
\begin{center}
\fbox{\begin{tikzpicture}[scale=0.35]
%Coordinates of the initial nodes of the Graph
\coordinate(X1) at (0.951786,4.693236);
\coordinate(X2) at (1,1);
\coordinate(X3) at (1,6);
\coordinate(X4) at (3,4);
\coordinate(X5) at (5.5,3);
\coordinate(X6) at (7.641752,3.554279);
\coordinate(X7) at (6.5,0);
\coordinate(X8) at (8,6);

\draw (0,5) circle (1.0);
\begin{scope}
\clip (0,0) rectangle (10,6); 
\draw (10,4) circle (2.4);
\end{scope}

%%2-median N
\draw[gray] (X2)--(X1);
\draw[gray] (X3)--(X1);
\draw[gray] (X4)--(X1);
\draw[gray] (X5)--(X6);
\draw[gray] (X7)--(X6);
\draw[gray] (X8)--(X6);
%%Points
%\fill (X1) circle (3pt);
\node[fill,star,star points=5, scale=0.5] at (X1) {};
\fill (X2) circle (2pt);
\fill (X3) circle (2pt);
\fill (X4) circle (2pt);
\fill (X5) circle (2pt);
%\fill (X6) circle (3pt);
\node[fill,star,star points=5, scale=0.5] at (X6) {};
\fill (X7) circle (2pt);
\fill (X8) circle (2pt);
\end{tikzpicture}}
\fbox{\begin{tikzpicture}[scale=0.35]
%Coordinates of the initial nodes of the Graph
\coordinate(X1) at (0,5);
\coordinate(X2) at (1,1);
\coordinate(X3) at (1,6);
\coordinate(X4) at (2.655601,3.577273);
\coordinate(X5) at (7.249924,3.000019);
\coordinate(X6) at (10,4);
\coordinate(X7) at (6.5,0);
\coordinate(X8) at (8,6);

\draw (3,4) circle (0.6);
\draw (5,3) circle (2.4);

%%2-center
\draw[gray] (X1)--(X4);
\draw[gray] (X2)--(X4);
\draw[gray] (X3)--(X4);
\draw[gray] (X6)--(X5);
\draw[gray] (X7)--(X5);
\draw[gray] (X8)--(X5);
%%Points
\fill (X1) circle (2pt);
\fill (X2) circle (2pt);
\fill (X3) circle (2pt);
%\fill (X4) circle (3pt);
\node[fill,star,star points=5, scale=0.5] at (X4) {};
%\fill (X5) circle (3pt);
\node[fill,star,star points=5, scale=0.5] at (X5) {};
\fill (X6) circle (2pt);
\fill (X7) circle (2pt);
\fill (X8) circle (2pt);
\end{tikzpicture}}
\fbox{\begin{tikzpicture}[scale=0.35]
%Coordinates of the initial nodes of the Graph
\coordinate(X1) at (0,5);
\coordinate(X2) at (1,1);
\coordinate(X3) at (1,6);
\coordinate(X4) at (2.431780,3.807320);
\coordinate(X5) at (7.366672,3.366584);
\coordinate(X6) at (10,4);
\coordinate(X7) at (6.5,0);
\coordinate(X8) at (8,6);

\draw (3,4) circle (0.6);
\draw (5,3) circle (2.4);
%%2-4center
\draw[gray] (X1)--(X4);
\draw[gray] (X2)--(X4);
\draw[gray] (X3)--(X4);
\draw[gray] (X6)--(X5);
\draw[gray] (X7)--(X5);
\draw[gray] (X8)--(X5);
%%Points
\fill (X1) circle (2pt);
\fill (X2) circle (2pt);
\fill (X3) circle (2pt);
%\fill (X4) circle (3pt);
\node[fill,star,star points=5, scale=0.5] at (X4) {};
%\fill (X5) circle (3pt);
\node[fill,star,star points=5, scale=0.5] at (X5) {};
\fill (X6) circle (2pt);
\fill (X7) circle (2pt);
\fill (X8) circle (2pt);
\end{tikzpicture}}
\end{center}

%% file: radius1.tex
\fbox{\begin{scaletikzpicturetowidth}{0.47\textwidth}
\begin{tikzpicture}[scale=\tikzscale]

\coordinate(X1) at (0,5);
\coordinate(X2) at (1,1);
\coordinate(X3) at (1,6);
\coordinate(X4) at (3,4);
\coordinate(X5) at (5.5,3);
\coordinate(X6) at (10,4);
\coordinate(X7) at (6.5,0);
\coordinate(X8) at (8,6);

\draw[white] (5,6.5) circle (0.12);

\draw (X1) circle (0.1);
\draw (X2) circle (0.3);
\draw (X3) circle (0.2);
\draw (X4) circle (0.4);
\draw (X5) circle (0.4);
\draw (X6) circle (0.2);
\draw (X7) circle (0.3);
\draw (X8) circle (0.1);

\fill (X1) circle (1pt);
\fill (X2) circle (1pt);
\fill (X3) circle (1pt);
\fill (X4) circle (1pt);
\fill (X5) circle (1pt);
\fill (X6) circle (1pt);
\fill (X7) circle (1pt);
\fill (X8) circle (1pt);

\end{tikzpicture}
\end{scaletikzpicturetowidth}}
\fbox{\begin{scaletikzpicturetowidth}{0.47\textwidth}
\begin{tikzpicture}[scale=\tikzscale]

\coordinate(X1) at (0,5);
\coordinate(X2) at (1,1);
\coordinate(X3) at (1,6);
\coordinate(X4) at (3,4);
\coordinate(X5) at (5.5,3);
\coordinate(X6) at (10,4);
\coordinate(X7) at (6.5,0);
\coordinate(X8) at (8,6);

\draw (X1) circle (0.5);
\draw (X2) circle (0.8);
\draw (X3) circle (0.8);
\draw (X4) circle (0.6);
\draw (X5) circle (0.9);
\draw (X6) circle (0.8);
\draw (X7) circle (0.6);
\draw (X8) circle (1);

\fill (X1) circle (1pt);
\fill (X2) circle (1pt);
\fill (X3) circle (1pt);
\fill (X4) circle (1pt);
\fill (X5) circle (1pt);
\fill (X6) circle (1pt);
\fill (X7) circle (1pt);
\fill (X8) circle (1pt);
\end{tikzpicture}
\end{scaletikzpicturetowidth}}

\fbox{\begin{scaletikzpicturetowidth}{0.47\textwidth}
\begin{tikzpicture}[scale=\tikzscale]

\coordinate(X1) at (0,5);
\coordinate(X2) at (1,1);
\coordinate(X3) at (1,6);
\coordinate(X4) at (3,4);
\coordinate(X5) at (5.5,3);
\coordinate(X6) at (10,4);
\coordinate(X7) at (6.5,0);
\coordinate(X8) at (8,6);

\draw[white] (5,6.5) circle (0.9);

\draw (X1) circle (1);
\draw (X2) circle (1.3);
\draw (X3) circle (1.2);
\draw (X4) circle (1.4);
\draw (X5) circle (1.4);
\draw (X6) circle (1.2);
\draw (X7) circle (1.3);
\draw (X8) circle (1.1);

\fill (X1) circle (1pt);
\fill (X2) circle (1pt);
\fill (X3) circle (1pt);
\fill (X4) circle (1pt);
\fill (X5) circle (1pt);
\fill (X6) circle (1pt);
\fill (X7) circle (1pt);
\fill (X8) circle (1pt);
\end{tikzpicture}
\end{scaletikzpicturetowidth}}
\fbox{\begin{scaletikzpicturetowidth}{0.47\textwidth}
\begin{tikzpicture}[scale=\tikzscale]
\coordinate(X1) at (0,5);
\coordinate(X2) at (1,1);
\coordinate(X3) at (1,6);
\coordinate(X4) at (3,4);
\coordinate(X5) at (5.5,3);
\coordinate(X6) at (10,4);
\coordinate(X7) at (6.5,0);
\coordinate(X8) at (8,6);

\draw (X1) circle (1.6);
\draw (X2) circle (1.9);
\draw (X3) circle (2);
\draw (X4) circle (1.8);
\draw (X5) circle (2);
\draw (X6) circle (1.7);
\draw (X7) circle (1.5);
\draw (X8) circle (1.9);

\fill (X1) circle (1pt);
\fill (X2) circle (1pt);
\fill (X3) circle (1pt);
\fill (X4) circle (1pt);
\fill (X5) circle (1pt);
\fill (X6) circle (1pt);
\fill (X7) circle (1pt);
\fill (X8) circle (1pt);
\end{tikzpicture}
\end{scaletikzpicturetowidth}}

%% file: usr1.tex
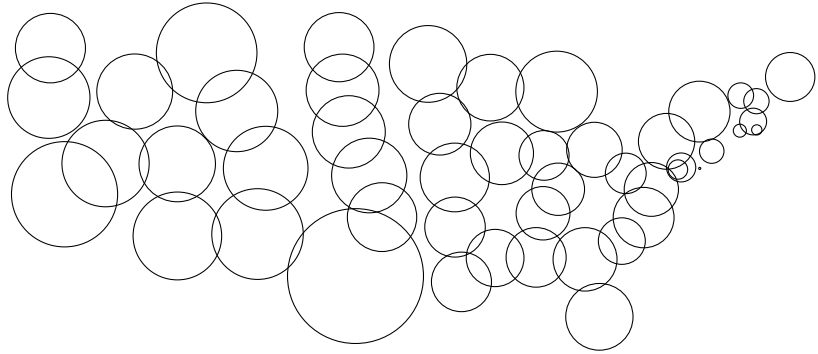
\begin{figure}[h]
\begin{center}
\begin{tikzpicture}[scale=0.19]
%Coordinates of the initial nodes of the Graph

\coordinate (X1) at (33.7296,32.7794);
\coordinate (X2) at (8.8981,34.2744);
\coordinate (X3) at (28.1157,34.8938);
\coordinate (X4) at (1.0887,37.1841);
\coordinate (X5) at (15.0105,38.9972);
\coordinate (X6) at (43.5436,38.9101);
\coordinate (X7) at (47.831,41.6219);
\coordinate (X8) at (45.0533,38.9896);
\coordinate (X9) at (38.1086,28.6305);
\coordinate (X10) at (37.1157,32.6415);
\coordinate (X11) at (5.9453,44.3509);
\coordinate (X12) at (31.3618,40.0417);
\coordinate (X13) at (34.2767,39.8942);
\coordinate (X14) at (27.0623,42.0751);
\coordinate (X15) at (22.1779,38.4937);
\coordinate (X16) at (35.2562,37.5347);
\coordinate (X17) at (28.5615,31.0689);
\coordinate (X18) at (51.3155,45.3695);
\coordinate (X19) at (43.7674,39.055);
\coordinate (X20) at (48.75,42.2596);
\coordinate (X21) at (35.1481,44.3467);
\coordinate (X22) at (26.253,46.2807);
\coordinate (X23) at (30.8905,32.7364);
\coordinate (X24) at (28.1003,38.3566);
\coordinate (X25) at (10.925,47.0527);
\coordinate (X26) at (20.7632,41.5378);
\coordinate (X27) at (3.9271,39.3289);
\coordinate (X28) at (48.9772,43.6805);
\coordinate (X29) at (45.8855,40.1907);
\coordinate (X30) at (14.4457,34.4071);
\coordinate (X31) at (45.0315,42.9538);
\coordinate (X32) at (41.1706,35.5557);
\coordinate (X33) at (20.0924,47.4501);
\coordinate (X34) at (37.7646,40.2862);
\coordinate (X35) at (23.064,35.5889);
\coordinate (X36) at (0,43.9336);
\coordinate (X37) at (42.7587,40.8781);
\coordinate (X38) at (49.0021,41.6762);
\coordinate (X39) at (39.6619,33.9169);
\coordinate (X40) at (20.332,44.4443);
\coordinate (X41) at (34.2078,35.858);
\coordinate (X42) at (21.2271,31.4757);
\coordinate (X43) at (8.888,39.3055);
\coordinate (X44) at (47.8925,44.0687);
\coordinate (X45) at (41.7046,37.5215);
\coordinate (X46) at (0.1111,47.3826);
\coordinate (X47) at (39.9356,38.6409);
\coordinate (X48) at (30.5642,44.6243);
\coordinate (X49) at (13.0071,42.9957);

\draw (X1) circle (2.078845);
\draw (X2) circle (3.065549);
\draw (X3) circle (2.093835);
\draw (X4) circle (3.673593);
\draw (X5) circle (2.929448);
\draw (X6) circle (0.676016);
\draw (X7) circle (0.452971);
\draw (X8) circle (0.075061);
\draw (X9) circle (2.328347);
\draw (X10) circle (2.213393);
\draw (X11) circle (2.624804);
\draw (X12) circle (2.18506);
\draw (X13) circle (1.732769);
\draw (X14) circle (2.153889);
\draw (X15) circle (2.604455);
\draw (X16) circle (1.825186);
\draw (X17) circle (2.078018);
\draw (X18) circle (1.707855);
\draw (X19) circle (1.011317);
\draw (X20) circle (0.932809);
\draw (X21) circle (2.823694);
\draw (X22) circle (2.677155);
\draw (X23) circle (1.998203);
\draw (X24) circle (2.397241);
\draw (X25) circle (3.481699);
\draw (X26) circle (2.525213);
\draw (X27) circle (3.019231);
\draw (X28) circle (0.877927);
\draw (X29) circle (0.847994);
\draw (X30) circle (3.166089);
\draw (X31) circle (2.12076);
\draw (X32) circle (2.106408);
\draw (X33) circle (2.41423);
\draw (X34) circle (1.922372);
\draw (X35) circle (2.400539);
\draw (X36) circle (2.847895);
\draw (X37) circle (1.948538);
\draw (X38) circle (0.356869);
\draw (X39) circle (1.624758);
\draw (X40) circle (2.521423);
\draw (X41) circle (1.863987);
\draw (X42) circle (4.7057);
\draw (X43) circle (2.645574);
\draw (X44) circle (0.890383);
\draw (X45) circle (1.877887);
\draw (X46) circle (2.424447);
\draw (X47) circle (1.41336);
\draw (X48) circle (2.323715);
\draw (X49) circle (2.839701);
\end{tikzpicture}
\end{center}
\caption{Basic neighborhoods for the 49 US states (radii $r^0$). \label{us}}
\end{figure}

%% file: usr1-solC.tex
\begin{figure}[h]
\begin{center}
\fbox{\begin{tikzpicture}[scale=0.18]
%Coordinates of the initial nodes of the Graph
% \node[anchor=south west,inner sep=0] at (-7,16) {\includegraphics[height=10cm,width=.95\textwidth]{mapa2.png}};
   
\coordinate (X1) at (33.7296,32.7794);
\coordinate (X2) at (8.8981,34.2744);
\coordinate (X3) at (28.1157,34.8938);
\coordinate (X4) at (1.0887,37.1841);
\coordinate (X5) at (15.0105,38.9972);
\coordinate (X6) at (43.5436,38.9101);
\coordinate (X7) at (47.831,41.6219);
\coordinate (X8) at (45.0533,38.9896);
\coordinate (X9) at (38.1086,28.6305);
\coordinate (X10) at (37.1157,32.6415);
\coordinate (X11) at (5.9453,44.3509);
\coordinate (X12) at (31.3618,40.0417);
\coordinate (X13) at (34.2767,39.8942);
\coordinate (X14) at (27.0623,42.0751);
\coordinate (X15) at (22.1779,38.4937);
\coordinate (X16) at (35.2562,37.5347);
\coordinate (X17) at (28.5615,31.0689);
\coordinate (X18) at (51.3155,45.3695);
\coordinate (X19) at (43.7674,39.055);
\coordinate (X20) at (48.75,42.2596);
\coordinate (X21) at (35.1481,44.3467);
\coordinate (X22) at (26.253,46.2807);
\coordinate (X23) at (30.8905,32.7364);
\coordinate (X24) at (28.1003,38.3566);
\coordinate (X25) at (10.925,47.0527);
\coordinate (X26) at (20.7632,41.5378);
\coordinate (X27) at (3.9271,39.3289);
\coordinate (X28) at (48.9772,43.6805);
\coordinate (X29) at (45.8855,40.1907);
\coordinate (X30) at (14.4457,34.4071);
\coordinate (X31) at (45.0315,42.9538);
\coordinate (X32) at (41.1706,35.5557);
\coordinate (X33) at (20.0924,47.4501);
\coordinate (X34) at (37.7646,40.2862);
\coordinate (X35) at (23.064,35.5889);
\coordinate (X36) at (0,43.9336);
\coordinate (X37) at (42.7587,40.8781);
\coordinate (X38) at (49.0021,41.6762);
\coordinate (X39) at (39.6619,33.9169);
\coordinate (X40) at (20.332,44.4443);
\coordinate (X41) at (34.2078,35.858);
\coordinate (X42) at (21.2271,31.4757);
\coordinate (X43) at (8.888,39.3055);
\coordinate (X44) at (47.8925,44.0687);
\coordinate (X45) at (41.7046,37.5215);
\coordinate (X46) at (0.1111,47.3826);
\coordinate (X47) at (39.9356,38.6409);
\coordinate (X48) at (30.5642,44.6243);
\coordinate (X49) at (13.0071,42.9957);

\draw (X8) circle (0.075061);
\coordinate (Y8) at (44.98868,39.02092);
\draw (X11) circle (2.624804);
\coordinate (Y11) at (7.082827,42.5325);
\draw (X16) circle (1.825186);
\coordinate (Y16) at (34.33381,36.62618);
\draw (X26) circle (2.525213);
\coordinate (Y26) at (21.10122,39.96479);
\draw (X38) circle (0.356869);
\coordinate (Y38) at (48.68806,41.66722);

\draw (X8) node[below=0.07cm]  {DC};
\draw (X11) node  {ID};
\draw (X16) node  {KY};
\draw (X26) node  {NE};
\draw (X38) node[right=0.07cm]  {RI};

\fill (Y8) circle (0.25);
\fill (Y11) circle (0.25);
\fill (Y16) circle (0.25);
\fill (Y26) circle (0.25);
\fill (Y38) circle (0.25);

\draw (X1)   node {\scriptsize AL};
\draw (X2)   node {\scriptsize AZ};
\draw (X3)   node {\scriptsize AR};
\draw (X4)   node {\scriptsize CA};
\draw (X5)   node {\scriptsize CO};
\draw (X6)   node {\scriptsize CT};
\draw (X7)   node[left=0.07cm] {\scriptsize DE};
%\draw (X8)   node[below right=0.07cm] {\scriptsize DC};
\draw (X9)   node {\scriptsize FL};
\draw (X10)  node {\scriptsize GA};
%\draw (X11)  node {\scriptsize ID  };
\draw (X12)  node {\scriptsize IL   };
\draw (X13)  node {\scriptsize IN  };
\draw (X14)  node {\scriptsize IA  };
\draw (X15)  node {\scriptsize KS };
%\draw (X16)  node {\scriptsize KY };
\draw (X17)  node {\scriptsize LA  };
\draw (X18)  node {\scriptsize ME };
\draw (X19)  node[right=0.07cm] {\scriptsize MD };
\draw (X20)  node {\scriptsize MA};
\draw (X21)  node {\scriptsize MI   };
\draw (X22)  node {\scriptsize MN };
\draw (X23)  node {\scriptsize MS };
\draw (X24)  node {\scriptsize MO};
\draw (X25)  node {\scriptsize MT };
%\draw (X26)  node {\scriptsize NE };
\draw (X27)  node {\scriptsize NV  };
\draw (X28)  node {\scriptsize NH };
\draw (X29)  node {\scriptsize NJ  };
\draw (X30)  node {\scriptsize NM};
\draw (X31)  node {\scriptsize NY  };
\draw (X32)  node {\scriptsize NC };
\draw (X33)  node {\scriptsize ND };
\draw (X34)  node {\scriptsize OH };
\draw (X35)  node {\scriptsize OK };
\draw (X36)  node {\scriptsize OR };
\draw (X37)  node {\scriptsize PA };
%\draw (X38)  node {\scriptsize RI  };
\draw (X39)  node {\scriptsize SC };
\draw (X40)  node {\scriptsize SD };
\draw (X41)  node {\scriptsize TN };
\draw (X42)  node {\scriptsize TX };
\draw (X43)  node {\scriptsize UT };
\draw (X44)  node[above=0.07cm] {\scriptsize VT };
\draw (X45)  node {\scriptsize VA  };
\draw (X46)  node {\scriptsize WA };
\draw (X47)  node {\scriptsize WV };
\draw (X48)  node {\scriptsize WI  };
\draw (X49)  node {\scriptsize WY  };

\draw (X1) -- (Y16);
\draw (X2) -- (Y11);
\draw (X3) -- (Y26);
\draw (X4) -- (Y11);
\draw (X5) -- (Y26);
\draw (X6) -- (Y8);
\draw (X7) -- (Y8);
\draw (X9) -- (Y16);
\draw (X10) -- (Y16);
\draw (X12) -- (Y16);
\draw (X13) -- (Y16);
\draw (X14) -- (Y26);
\draw (X15) -- (Y26);
\draw (X17) -- (Y16);
\draw (X18) -- (Y38);
\draw (X19) -- (Y8);
\draw (X20) -- (Y38);
\draw (X21) -- (Y16);
\draw (X22) -- (Y26);
\draw (X23) -- (Y16);
\draw (X24) -- (Y26);
\draw (X25) -- (Y11);
\draw (X27) -- (Y11);
\draw (X28) -- (Y38);
\draw (X29) -- (Y8);
\draw (X30) -- (Y26);
\draw (X31) -- (Y8);
\draw (X32) -- (Y8);
\draw (X33) -- (Y26);
\draw (X34) -- (Y16);
\draw (X35) -- (Y26);
\draw (X36) -- (Y11);
\draw (X37) -- (Y8);
\draw (X39) -- (Y16);
\draw (X40) -- (Y26);
\draw (X41) -- (Y16);
\draw (X42) -- (Y26);
\draw (X43) -- (Y11);
\draw (X44) -- (Y38);
\draw (X45) -- (Y8);
\draw (X46) -- (Y11);
\draw (X47) -- (Y8);
\draw (X48) -- (Y16);
\draw (X49) -- (Y11);

\draw (1,25) node {$f^* = 16.24907$};

\end{tikzpicture}}\end{center}
\caption{Solutions for $5$-center problem under Scenario SC1 for the US data set.\label{fig:solC5}}
\end{figure}
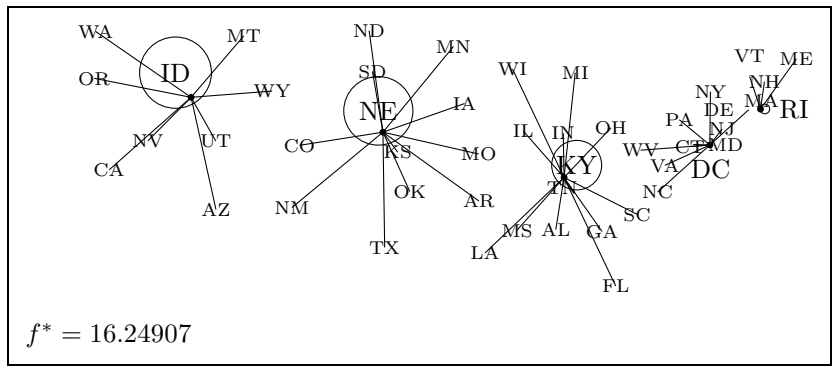

%% file: usr1-solC0.tex
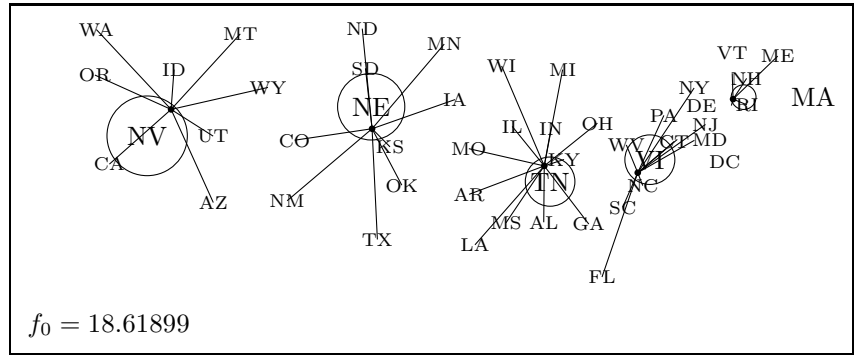
\begin{figure}[h]
\begin{center}
\fbox{\begin{tikzpicture}[scale=0.175]
%Coordinates of the initial nodes of the Graph
% \node[anchor=south west,inner sep=0] at (-7,16) {\includegraphics[height=10cm,width=.95\textwidth]{mapa2.png}};
   
\coordinate (X1) at (33.7296,32.7794);
\coordinate (X2) at (8.8981,34.2744);
\coordinate (X3) at (28.1157,34.8938);
\coordinate (X4) at (1.0887,37.1841);
\coordinate (X5) at (15.0105,38.9972);
\coordinate (X6) at (43.5436,38.9101);
\coordinate (X7) at (47.831,41.6219);
\coordinate (X8) at (45.0533,38.9896);
\coordinate (X9) at (38.1086,28.6305);
\coordinate (X10) at (37.1157,32.6415);
\coordinate (X11) at (5.9453,44.3509);
\coordinate (X12) at (31.3618,40.0417);
\coordinate (X13) at (34.2767,39.8942);
\coordinate (X14) at (27.0623,42.0751);
\coordinate (X15) at (22.1779,38.4937);
\coordinate (X16) at (35.2562,37.5347);
\coordinate (X17) at (28.5615,31.0689);
\coordinate (X18) at (51.3155,45.3695);
\coordinate (X19) at (43.7674,39.055);
\coordinate (X20) at (48.75,42.2596);
\coordinate (X21) at (35.1481,44.3467);
\coordinate (X22) at (26.253,46.2807);
\coordinate (X23) at (30.8905,32.7364);
\coordinate (X24) at (28.1003,38.3566);
\coordinate (X25) at (10.925,47.0527);
\coordinate (X26) at (20.7632,41.5378);
\coordinate (X27) at (3.9271,39.3289);
\coordinate (X28) at (48.9772,43.6805);
\coordinate (X29) at (45.8855,40.1907);
\coordinate (X30) at (14.4457,34.4071);
\coordinate (X31) at (45.0315,42.9538);
\coordinate (X32) at (41.1706,35.5557);
\coordinate (X33) at (20.0924,47.4501);
\coordinate (X34) at (37.7646,40.2862);
\coordinate (X35) at (23.064,35.5889);
\coordinate (X36) at (0,43.9336);
\coordinate (X37) at (42.7587,40.8781);
\coordinate (X38) at (49.0021,41.6762);
\coordinate (X39) at (39.6619,33.9169);
\coordinate (X40) at (20.332,44.4443);
\coordinate (X41) at (34.2078,35.858);
\coordinate (X42) at (21.2271,31.4757);
\coordinate (X43) at (8.888,39.3055);
\coordinate (X44) at (47.8925,44.0687);
\coordinate (X45) at (41.7046,37.5215);
\coordinate (X46) at (0.1111,47.3826);
\coordinate (X47) at (39.9356,38.6409);
\coordinate (X48) at (30.5642,44.6243);
\coordinate (X49) at (13.0071,42.9957);

\draw (X20) circle (0.932809);
\coordinate (Y20) at (47.95932,42.10766);
\draw (X26) circle (2.525213);
\coordinate (Y26) at (20.82979,39.86614);
\draw (X27) circle (3.019231);
\coordinate (Y27) at (5.727338,41.32962);
\draw (X41) circle (1.863987);
\coordinate (Y41) at (33.76884,37.04855);
\draw (X45) circle (1.877887);
\coordinate (Y45) at (40.79457,36.53385);

\draw (X20) node[right=0.5cm]  {MA};
\draw (X26) node  {NE};
\draw (X27) node  {NV};
\draw (X41) node  {TN};
\draw (X45) node  {VI};

\fill (Y20) circle (0.25);
\fill (Y26) circle (0.25);
\fill (Y27) circle (0.25);
\fill (Y41) circle (0.25);
\fill (Y45) circle (0.25);

\draw (X1)   node {\scriptsize AL};
\draw (X2)   node {\scriptsize AZ};
\draw (X3)   node {\scriptsize AR};
\draw (X4)   node {\scriptsize CA};
\draw (X5)   node {\scriptsize CO};
\draw (X6)   node {\scriptsize CT};
\draw (X7)   node[left=0.07cm] {\scriptsize DE};
\draw (X8)   node[below right=0.07cm] {\scriptsize DC};
\draw (X9)   node {\scriptsize FL};
\draw (X10)  node {\scriptsize GA};
\draw (X11)  node {\scriptsize ID  };
\draw (X12)  node {\scriptsize IL   };
\draw (X13)  node {\scriptsize IN  };
\draw (X14)  node {\scriptsize IA  };
\draw (X15)  node {\scriptsize KS };
\draw (X16)  node {\scriptsize KY };
\draw (X17)  node {\scriptsize LA  };
\draw (X18)  node {\scriptsize ME };
\draw (X19)  node[right=0.07cm] {\scriptsize MD };
%\draw (X20)  node {\scriptsize MA};
\draw (X21)  node {\scriptsize MI   };
\draw (X22)  node {\scriptsize MN };
\draw (X23)  node {\scriptsize MS };
\draw (X24)  node {\scriptsize MO};
\draw (X25)  node {\scriptsize MT };
%\draw (X26)  node {\scriptsize NE };
%\draw (X27)  node {\scriptsize NV  };
\draw (X28)  node {\scriptsize NH };
\draw (X29)  node {\scriptsize NJ  };
\draw (X30)  node {\scriptsize NM};
\draw (X31)  node {\scriptsize NY  };
\draw (X32)  node {\scriptsize NC };
\draw (X33)  node {\scriptsize ND };
\draw (X34)  node {\scriptsize OH };
\draw (X35)  node {\scriptsize OK };
\draw (X36)  node {\scriptsize OR };
\draw (X37)  node {\scriptsize PA };
\draw (X38)  node {\scriptsize RI  };
\draw (X39)  node {\scriptsize SC };
\draw (X40)  node {\scriptsize SD };
%\draw (X41)  node {\scriptsize TN };
\draw (X42)  node {\scriptsize TX };
\draw (X43)  node {\scriptsize UT };
\draw (X44)  node[above=0.07cm] {\scriptsize VT };
%\draw (X45)  node {\scriptsize VA  };
\draw (X46)  node {\scriptsize WA };
\draw (X47)  node {\scriptsize WV };
\draw (X48)  node {\scriptsize WI  };
\draw (X49)  node {\scriptsize WY       };

\draw (1,25) node {$f_0 = 18.61899$};

\draw (X1) -- (Y41);
\draw (X2) -- (Y27);
\draw (X3) -- (Y41);
\draw (X4) -- (Y27);
\draw (X5) -- (Y26);
\draw (X6) -- (Y45);
\draw (X7) -- (Y20);
\draw (X8) -- (Y45);
\draw (X9) -- (Y45);
\draw (X10) -- (Y41);
\draw (X11) -- (Y27);
\draw (X12) -- (Y41);
\draw (X13) -- (Y41);
\draw (X14) -- (Y26);
\draw (X15) -- (Y26);
\draw (X16) -- (Y41);
\draw (X17) -- (Y41);
\draw (X18) -- (Y20);
\draw (X19) -- (Y45);
\draw (X21) -- (Y41);
\draw (X22) -- (Y26);
\draw (X23) -- (Y41);
\draw (X24) -- (Y41);
\draw (X25) -- (Y27);
\draw (X28) -- (Y20);
\draw (X29) -- (Y45);
\draw (X30) -- (Y26);
\draw (X31) -- (Y45);
\draw (X32) -- (Y45);
\draw (X33) -- (Y26);
\draw (X34) -- (Y41);
\draw (X35) -- (Y26);
\draw (X36) -- (Y27);
\draw (X37) -- (Y45);
\draw (X38) -- (Y20);
\draw (X39) -- (Y45);
\draw (X40) -- (Y26);
\draw (X42) -- (Y26);
\draw (X43) -- (Y27);
\draw (X44) -- (Y20);
\draw (X46) -- (Y27);
\draw (X47) -- (Y45);
\draw (X48) -- (Y41);
\draw (X49) -- (Y27);

\end{tikzpicture}}\end{center}
\caption{Initial Solutions for $5$-center problem under Scenario SC1 for the US data set.\label{fig:solC05}}
\end{figure}